\documentclass[12pt]{amsart}
\usepackage{graphicx}
\usepackage{verbatim}
\usepackage{amssymb}
\usepackage{url}
\usepackage{vmargin,color}
\usepackage[latin1]{inputenc}
\newtheorem{thm}{Theorem}[section]
\newtheorem*{acknowledgement*}{Acknowledgement}
\newtheorem{cor}[thm]{Corollary}
\newtheorem{lem}[thm]{Lemma}
\newtheorem{ques}[thm]{Question}

\newtheorem{prop}[thm]{Proposition}
\theoremstyle{definition}
\newtheorem{defn}[thm]{Definition}
\theoremstyle{remark}
\newtheorem{rem}[thm]{Remark}
\numberwithin{equation}{section}

\newcommand{\set}[1]{\left\{#1\right\}}
\newcommand{\Real}{\mathbb R}
\newcommand{\R}{\mathbb R}

\newcommand{\func}[1]{\ensuremath{\mathrm{#1} \:} }

\newcommand{\Div}[0]{\func{div}}

\newcommand{\eE}[0]{\mathbf{e}}

\newcommand{\spt}[0]{\mathrm{spt}}
\newcommand{\reg}[0]{\mathrm{reg}}
\newcommand{\sing}[0]{\mathrm{sing}}
\newcommand{\Cat}[0]{\mathrm{Cat}}
\newcommand{\pa}[0]{\partial}
\newcommand{\Ha}{\mathcal{H}}
\newcommand{\Om}{\Omega}

\title[Symmetry and Rigidity of Minimal Surfaces with Singularities]{Symmetry and Rigidity of Minimal Surfaces \\ with Plateau-like Singularities}
\author{Jacob Bernstein}
\address{Department of Mathematics, Johns Hopkins University, 3400 N. Charles Street, Baltimore, MD 21218}
\email{bernstein@math.jhu.edu}
\author{Francesco Maggi}
\address{Department of Mathematics, The University of Texas at Austin, 2515 Speedway, Stop C1200, Austin TX 78712-1202, USA}
\email{maggi@math.utexas.edu}
\thanks{JB was partially supported by the NSF Grant  DMS-1609340 and DMS-1904674. FM was partially supported by the NSF Grants DMS-156535 and DMS-FRG-1854344.}

\begin{document}
\begin{abstract}
By employing the method of moving planes in a novel way we extend some classical symmetry and rigidity results for smooth minimal surfaces to surfaces that have singularities of the sort typically observed in soap films.
\end{abstract}
\maketitle


\section{Introduction}

\subsection{Overview} Minimal surfaces in $\R^3$ provide the standard mathematical model of soap films at equilibrium. Nevertheless, there is a historical mismatch between the classical theory of minimal surfaces, which focuses on smooth immersions with vanishing mean curvature, and the richer structures documented experimentally since the pioneering work of Plateau \cite{plateau}. Indeed, two types of singular points are observed in soap films, called $Y$ and $T$ points; see Figure \ref{fig yt} below. We call the surfaces described in experiments {\it minimal Plateau surfaces} and ask:

\medskip

{\it To what extent may the classical theory of minimal surfaces be generalized to minimal Plateau surfaces and what new conclusions may be drawn?}

\medskip

This paper studies this question in the model case provided by {\bf Schoen's rigidity theorem for catenoids} \cite{SchoenSymmetry}: a (classical) minimal surface in $\R^3$ spanning two parallel circles with centers on the same axis has rotational symmetry about this axis and so is either a pair of flat disks or a subset of a catenoid. Schoen's theorem is an interesting model case for two reasons: (i) its extension to minimal Plateau surfaces requires the inclusion of new cases of rigidity, given by singular catenoids;  (ii) Schoen's proof uses Alexandrov's method of moving planes \cite{alexandrov}, which has been almost exclusively applied in the smooth setting: thus its adaptation to a class containing singular surfaces is notable. The only other application of the moving planes method in a non-smooth setting that we are aware of is the recent work \cite{HaslhoferEtAl,HaHeWh}. However, in that work \emph{a posteriori} regularity is derived from the moving planes method despite allowing \emph{a priori} singularities. This is unlike our applications in which genuinely singular surfaces are symmetric examples; see
\begin{figure}
	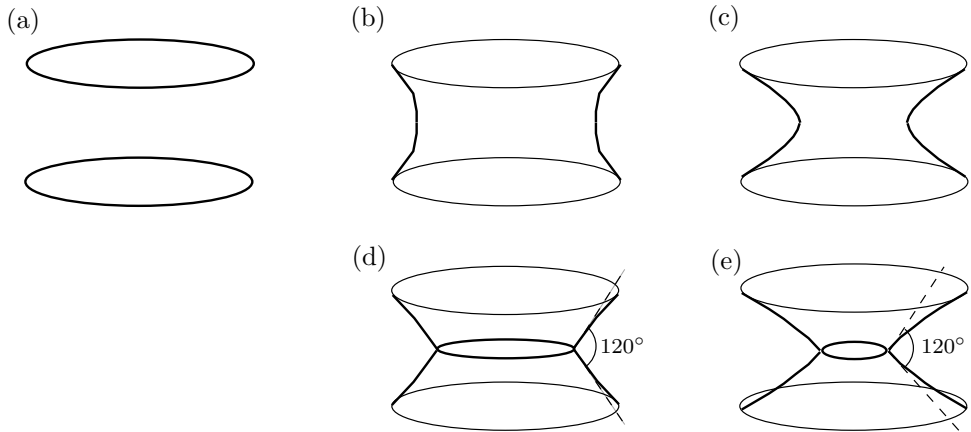\caption{{\small Two parallel circles with same radii lying at a sufficiently small distance span exactly three smooth minimal surfaces: a pair of disks, a ``fat'' catenoid (which is stable) and a ``skinny'' catenoid (which is unstable). The same circles span five minimal Plateau surfaces, the two new cases being defined by a pair of ``singular'' $Y$-catenoids.}}\label{fig catenoids}
\end{figure}
Figure \ref{fig catenoids}.

\medskip

This introduction is organized as follows. In Section \ref{section schoen rigidity} we recall the rigidity theorems from \cite{SchoenSymmetry}. In Sections \ref{section plateau surfaces} and \ref{section cell structure} we define Plateau surfaces and introduce a notion of orientability for them, that we call the cell structure condition. In Section \ref{section schoen plateau} we state our main results, which extend Schoen's rigidity theorems to minimal Plateau surfaces. Finally, in Sections \ref{section physical motivation} and  \ref{section nonsmooth rigidity} we discuss further the physical and mathematical motivations for Plateau surfaces and situate them within the more general frameworks provided by geometric measure theory.

\subsection{Schoen's rigidity theorems}\label{section schoen rigidity} The {\it first rigidity theorem} proved  in \cite{SchoenSymmetry} states that a minimal immersion of a compact connected surface with  boundary consisting of a pair of coaxial circles in parallel planes is, up to rigid motion and dilation, a piece of the catenoid.  Here the catenoid is the minimal surface
$$
\Cat=\set{x_1^2+x_2^2=\cosh^2 x_3}\,.
$$
The {\it second rigidity theorem} is more global in nature.  It says that, up to rigid motion and dilation, any complete minimal immersion that has two regular ends, must either be a catenoid or a pair of planes.
This means each end is modeled on either a catenoidal or planar end -- see Definition \ref{def regular end}. In both rigidity theorems the hypotheses that the minimal surface be an immersion is essential, as can be seen by the example of the \emph{Y-catenoid},
$$
\Cat_Y=\set{x_1^2+x_2^2=\cosh^2 (|x_3|+h_0)}\bigcup \set{3(x_1^2+x_2^2)\leq 4, x_3=0}
$$
(where $h_0=\log(3)/2$ is the unique solution to $\sinh h_0=1/\sqrt{3}$). Indeed, $\Cat_Y$ is minimal both in a distributional sense (that is, as a stationary $2$-dimensional varifold in $\R^3$) and is the prototypical example of what we call a minimal Plateau surface.

\subsection{Plateau surfaces}\label{section plateau surfaces}
Let $\mathcal{K}$ a family of cones in $\R^3$ with vertex the origin such that if $K_1,K_2\in\mathcal{K}$ and $K_1\ne K_2$ then $K_1\ne R(K_2)$ for every isometry $R$ of $\R^3$, and such that
\[
\{P,H\}\subset\mathcal{K}\,,
\]
where $P=\set{x_3=0}$ is a plane and $H=\set{x_3=0, x_1\geq 0}$ a half-plane. In particular, if $K\in\mathcal{K}\setminus\{P,H\}$, then $K$ is neither a plane nor a half-plane. Given $U\subset\Real^3$ open, a closed subset $\Sigma\subset U$ is a {\bf $\mathcal{K}$-surface in $U$} if, for some $\alpha\in(0,1)$ and for all $p\in \Sigma\cap U$, there are $r>0$ and a $C^{1,\alpha}$-regular diffeomorphism $\phi: B_{r}(p)\subset U\to \Real^3$ so that $\phi(\Sigma\cap B_{r}(p))\in\mathcal{K}$ and $D\phi_p  \in O(3)$, i.e., $D \phi_p$ is an orthogonal linear transformation. The element of $\mathcal{K}$ corresponding to $p\in\Sigma$ is unique
and is denoted by
\[
\hat{T}_p\Sigma\in \mathcal{K}\,.
\]
The {\bf tangent cone} of $\Sigma$ at $p$, denoted $T_p\Sigma$, is defined by  $D\phi_p\left(T_p\Sigma\right)= \hat{T}_p\Sigma$.  Clearly, $T_p\Sigma=\lim_{\rho\to 0^+} (\Sigma-p)/\rho$ where the limit is in the pointed Hausdorff sense.

For each $K\in\mathcal{K}$, we let $\Sigma_K=\{p\in\Sigma\cap U:\hat{T}_p\Sigma=K\}$. Correspondingly, we identify the sets of {\bf interior points} ${\rm int}(\Sigma)=\Sigma_P$, of {\bf boundary points} $\partial\Sigma=\Sigma_H$, of {\bf regular points} $\reg(\Sigma)={\rm int}(\Sigma)\cup\pa\Sigma$, and of {\bf singular points} $\sing(\Sigma)=\Sigma\setminus\reg(\Sigma)$. By construction, an H\"older continuous vector field $\nu^{{\rm co}}_\Sigma$ of {\bf outer unit conormals} to $\Sigma$ can be defined along $\pa\Sigma$. When $\sing(\Sigma)=\emptyset$, the notion of $\mathcal{K}$-surface reduces to that of {\bf regular surface} (with boundary and of class $C^{1,\alpha}$) in $U$.

\medskip

A (relatively) closed subset $\Sigma\subset U$, in an open subset $U\subset \Real^3$, is a {\bf Plateau surface in $U$} if: {\bf (a)} $\Sigma$ is a $\mathcal{K}$-surface in $U$ for
\begin{equation}
\label{K plateau}
\mathcal{K}=\{P,H,Y,T\}\,,
\end{equation}
where $Y=H\cup H_{120}\cup H_{-120}$ (and $H_\theta$ is the rotation of $H$ by $\theta$-degrees about the $x_2$-axis), and $T$ is the cone over the edges of a reference regular tetrahedron centered at the origin, see
\begin{figure}
	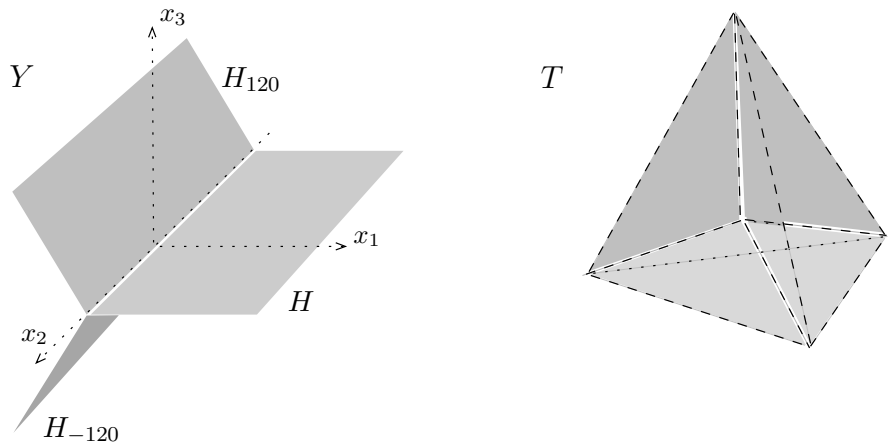\caption{{\small The model $Y$ and $T$ cones in $\mathcal{K}$. Here $H=\{x_3=0,x_1\ge 0\}$ and $H_\theta$ is obtained by rotating $H$ around the $x_2$-axis by $\theta$-degrees.}}\label{fig yt}
\end{figure}
Figure \ref{fig yt}; {\bf (b)} each connected component of $\mathrm{int}(\Sigma)$ has (weak) constant mean curvature. If $\mathrm{int}(\Sigma)$ has zero mean curvature, then $\Sigma$ is a {\bf minimal Plateau surface in $U$}. When $\Sigma_T=\emptyset$, one calls $\Sigma$ a {\bf $Y$-surface}.

\begin{rem}\label{rmk PS from c1a to analytic}
	{\rm If $\Sigma$ is a Plateau surface, then $\Sigma\backslash (\partial \Sigma\cup \Sigma_T)$ admits smooth (in fact real-analytic) charts. Indeed, standard elliptic regularity ensures the smoothness of any $C^{1,\alpha}$-graph whose mean curvature is constant in a weak sense and so $\mathrm{int}(\Sigma)$ consists of smooth surfaces.  Furthermore, work of Kinderleher, Nirenberg and Spruck \cite[Theorem 5.2]{KNS} implies that each component of $\Sigma_Y$ is a smooth curve.}
\end{rem}


\subsection{Orientability of Plateau surfaces}\label{section cell structure}
We introduce a notion of orientability in the Plateau setting that generalizes the notion of a regular surface separating an ambient three-manifold.
A Plateau surface $\Sigma$ {\bf defines a cell structure} in $U\subset\R^3$ open, if there exists a family of open, connected sets
$\mathcal{C}(\Sigma)=\set{U^i: 1\leq i \leq N}$, called the {\bf cells of $\Sigma$}, such that
\begin{equation}
\label{cell structure}
\partial \Sigma \subset \partial U\,,\qquad U\backslash \Sigma =\bigcup_{i=1}^N U^i
\end{equation}
and, for each $p\in \Sigma\cap U=\Sigma\setminus\partial\Sigma$ there is a $\rho>0$ so $B_{\rho}(p)\subset U$ and, for each $0<\rho'<\rho$ and $i=1, \ldots, N$,  $B_{\rho'}(p)\cap U^i$ is connected (possibly empty).

\medskip

Clearly,  $\Cat$  defines a cell structure in $\Real^3$ with two cells while $\Cat_Y$ defines a cell structure in $\Real^3$ with three cells. An example of a surface not defining a cell structure is illustrated in Figure \ref{fig regularbi}-(b).  A connected regular surface defines a cell structure in $U$ when it is separating in $U$.
Observe that the tetrahedral cone $T\subset \Real^3$ defines a cell structure in $\mathbb{R}^3$ but is not a flat chain mod 3.

\subsection{Schoen's rigidity theorems for Plateau surfaces}\label{section schoen plateau}
Let us recall some notation and terminology from \cite{SchoenSymmetry}. A set $\Sigma\subset \Real^3$ is a {\bf graph} if $\pi|_{\Sigma}:\Sigma \to \Real^2$ is one-to-one, where $\pi: \Real^3\to \Real^2$ is the projection $\pi((\mathbf{y},x_3))=\mathbf{y}$. We say that $\Sigma$ is a {\bf graph of locally bounded slope} if it is a graph and there exists a (one- or two-dimensional) $C^{1}$-submanifold $\sigma$ of $\R^3$ such that $\Sigma=\bar{\sigma}$ and such that $T_p \sigma$ is transverse to $\mathbf{e}_3$ for each $p\in \sigma$ -- for example, $\Sigma=\{x_3\ge0\,,x_1^2+x_2^2+x_3^2=1\}$ and $\Sigma=\{x_3\ge0\,, x_2=0\,,x_1^2+x_3^2=1\}$ are both graphs of locally bounded slope.

Given an open subset $\Omega\subset \mathbb{R}^2$, let
\[
C_{\Omega}=\set{(\mathbf{y}, z): \mathbf{y}\in \Omega}\subset \R^3
\]
be the cylinder over $\Omega$. A {\bf minimal Plateau bi-graph} over $\Omega$ is a (not necessarily connected) minimal Plateau surface, $\Sigma$, satisfying $\Sigma\subset\bar{C}_\Omega$,  $\partial\Sigma=\Sigma\cap\partial C_\Omega$, and so
\[
\Sigma_{0^+}=\Sigma\cap\{x_3\ge0\}\qquad\mbox{ and }\qquad\Sigma_{0^-}=\Sigma\cap\{x_3\le 0\}\,,
\]
are both graphs of locally bounded slope; see
\begin{figure}
	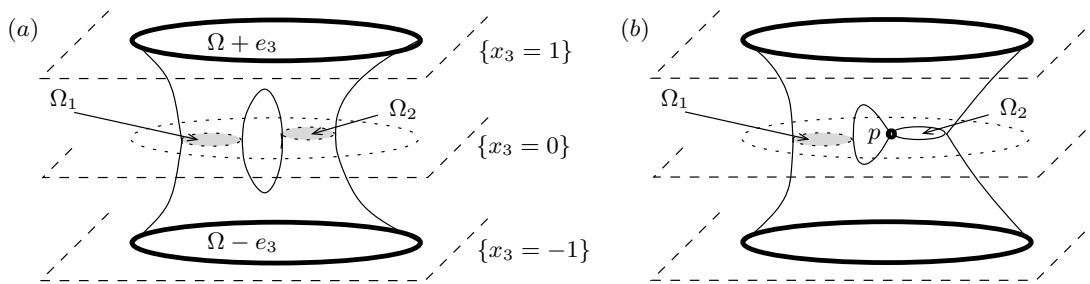\caption{\small{Illustrations of the definition of minimal Plateau bi-graph: (a) a regular minimal Plateau bi-graph that is not simple; (b) a non-simple minimal Plateau bi-graph $\Sigma$ with non-trivial singular set; notice that in this case $\Om_1$ is not part of $\Sigma$, but $\Om_2\subset\Sigma$, with $\partial\Omega_2=\sing(\Sigma)=\Sigma_Y$.}}\label{fig regularbi}
\end{figure}
Figure \ref{fig regularbi}. Clearly, such $\Sigma$ must have $\Sigma_T=\emptyset$, $\Sigma_Y\subset\{x_3=0\}$, and if $p\in\Sigma_Y$, then the spine of $T_p\Sigma$ is contained in $\{x_3=0\}$.   If, in addition, $\Sigma\cap\{x_3=0\}$ is empty or is the boundary of a single topological disk contained in $\{x_3=0\}$, then $\Sigma$ is {\bf simple}. For instance, $\Cat\cap \bar{C}_{B_R}$,  $ \set{|x_3|=1}\cap \bar{C}_R$ and $\Cat_Y\cap \bar{C}_{B_R}$ are all simple minimal Plateau bi-graphs for appropriate $R$. Simple minimal Plateau bi-graphs define a cell structure in $C_{\Omega}$, but this is not necessarily the case when $\Sigma$ is not simple; see
Figure \ref{fig regularbi}-(b).

Our extension of Schoen's first rigidity result to the Plateau setting is as follows.
\begin{thm}\label{MainThm}
	Let $\Omega\subset \Real^2$ be a bounded, open convex set with $C^1$-boundary, and let $\Sigma$ be a compact, minimal Plateau surface in $\R^3$ with
	\[
	\partial\Sigma=(\partial\Omega)\times\{1,-1\}\,.
	\]
	If $\Sigma$ defines a cell structure in $U=\set{|x_3|<1}$, then $\Sigma$ is a simple minimal Plateau bi-graph, which is symmetric by reflection through $\{x_3=0\}$. Moreover, if $\Omega$ is the interior of a circle, then $\Sigma$ is either a union of two disks, or, up to translation and dilation,  is a subset of $\Cat$ or of $\Cat_Y$.
\end{thm}

\begin{rem}
	Unlike Schoen's first result, our proof does not apply to arbitrary pairs of coaxial circles.  However, we expect the more general result is also true.
\end{rem}

We also obtain an analog of Schoen's second rigidity theorem.  Namely, global rigidity and symmetry for minimal Plateau surfaces with two regular ends that are subject to the same orientability condition used in the previous theorem. The precise definition of regular end is given later on in Definition \ref{def regular end}.

\begin{thm}\label{GlobalThm}
	Let $\Sigma$ be a minimal Plateau surface that defines a cell structure in $\Real^3$.  If there is an $R_0>0$ so that $\Sigma\backslash B_{R_0}$ has two regular ends, then, up to a rigid motion and dilation, $\Sigma$ is either a pair of planes, a catenoid or a $Y$-catenoid.
\end{thm}

\begin{rem}
	It is unclear whether the assumption that the minimal Plateau surfaces define cell structures in Theorems \ref{MainThm} and \ref{GlobalThm} are necessary or just a technical hypothesis needed for our proof. This point is further discussed in Section \ref{Questions}.
\end{rem}

\subsection{Physical and mathematical motivation}\label{section physical motivation} The physical motivation for the notion of Plateau surface proposed in this paper lies in the celebrated Plateau's laws, which are empirical observations about the geometric structure of soap films. Plateau's laws state that soap films at equilibrium are arranged into smooth surfaces with constant mean curvature, meeting in threes along edges at 120$^\circ$ degrees angles; and that these edges meet in four at vertex points, and they do so at the angles defined by the skeleton of a regular tetrahedron. The definition given in Section \ref{section plateau surfaces} simply captures, in exact mathematical terms, all the features listed in Plateau's laws -- as explained in Remark \ref{rmk PS from c1a to analytic}, the $C^{1,\alpha}$-regularity requirement is purely technical. Thus Plateau surfaces match Plateau's description of soap films arising in clusters of soap bubbles, while minimal Plateau surfaces correspond to soap films spanning a fixed ``wire frame".

\medskip

The mathematical justification for our definition of Plateau surface is given by Taylor's theorem \cite{taylor76}. Indeed, Taylor proved that if $U\subset\R^3$ is open, $\Sigma$ is a relatively compact and rectifiable set in $U$, $\Sigma=U\cap\spt(\mathcal{H}^2\llcorner \Sigma)$, and, for some $\alpha>2$,
\begin{equation}
\label{taylor hp}
\mathcal{H}^2(\Sigma)\le\mathcal{H}^2(\varphi(\Sigma))+C\,r^\alpha
\end{equation}
whenever $\{\varphi\ne{\rm id}\}\subset B_r(x)\subset\subset U$, $x\in\Sigma$ and ${\rm Lip}\,\varphi<\infty$, then, in our terminology, $\Sigma$ is a $\mathcal{K}$-surface without boundary in $U$ where $\mathcal{K}$ is as in \eqref{K plateau}.  Moreover, when $C=0$, $\Sigma$ is a minimal Plateau surface.

The significance of Taylor's theorem is that it explains the (interior) singularities observed by Plateau solely in terms of the geometric calculus of variations. Various ``non-distributional approaches to Plateau's problem'' have been proposed to show the existence of compact sets $\Sigma$ satisfying \eqref{taylor hp} with $C=0$ and with $U$ given by the complement of a compact ``wire frame'': these include, at least, Reifenberg's approach of homological spanning conditions, the Harrison-Pugh approach of homotopic spanning conditions, and David's notion of sliding minimizers; see \cite{reifenberg1,reifenberg2,reifenberg3,davidshouldwe,fangAPISA,FangKola,harrisonpughACV,harrisonpughGENMETH,DLGM,delederosaghira} and other related papers.  These approaches provide rigorous constructions of many minimal Plateau surfaces -- though care has to be taken at the boundary; see below.

\subsection{Rigidity theorems in more general non-smooth settings}\label{section nonsmooth rigidity}  Plateau surfaces provide an interesting ``semi-classical'' setting for extending the theory of minimal surfaces. The same goal could however be pursued in even more general settings -- specifically those provided by geometric measure theory (GMT).  There are two major motivations for this.

\medskip

First of all, the two-dimensional area minimizing surfaces in $\R^3$ constructed by the non-distributional approaches to Plateau's problem mentioned above (e.g., \cite{reifenberg1,harrisonpughACV,DLGM,davidshouldwe}), as well as those found in distributional approaches (e.g., flat chains modulo 3 \cite{taylor73}), may possess boundary singularities. This is not a purely theoretical issue as boundary singularities are also observed in physical soap films. Thus, it is natural to consider a more general notion of Plateau surface where boundary behavior is not modeled only by the half-plane $H$, but by more general cones.  In particular, Plateau surfaces as introduced here should be properly understood as ``Plateau surfaces with regular boundary''.   A list of possible boundary singularities is described in \cite[Section 5.2 and Figure 5.3]{lawlormorgan96}, although not all the examples in that list are likely to be locally area minimizing (i.e., physical), and so it is unclear what the correct modification of the definition adopted in this paper should be. By working in the language of GMT one sidesteps this difficulty by working in a class large enough to encompass all possible boundary singularities.

\medskip

Secondly,  GMT provides powerful compactness theorems which, in turn, allow one to turn rigidity theorems like Theorem \ref{MainThm} into interesting perturbative results. For instance, in the case of the volume-preserving mean curvature flow, a characterization of equilibrium states requires the generalization of the classical Alexandrov's theorem (smooth boundaries with constant mean curvature enclosing finite volumes are spheres \cite{alexandrov}) to the class of sets of finite perimeter and finite volume with constant distributional mean curvature; see \cite{delgadinomaggiAPDE}.  In a similar vein, Theorem \ref{MainThm} could be used to understand the long time behavior of (singular) mean curvature flows with fixed boundary given by two parallel convex curves; see \cite{StuvardTonegawa}.

\medskip

With these motivations in mind, in the follow-up paper \cite{VarifoldPaper} we extend the reach of our rigidity theorems from minimal Plateau surfaces to an appropriate class of stationary varifolds.

\subsection{Organization of the paper} In Section \ref{section moving planes} we present the key technical statement of the paper, Theorem \ref{CylThm}. Sections \ref{section rigidity slab} and \ref{section global rigidity} contain, respectively, the proofs of Theorem \ref{MainThm} and Theorem \ref{GlobalThm}, while in Section \ref{Questions} we collect some open questions.

\section{Moving planes for minimal Plateau surfaces in  a cylinder}\label{section moving planes} In Section \ref{section rs and uc} we prove a removable singularity result and a unique continuation principle for minimal Plateau surfaces and record a  simple observation about the infinitesimal structure of cellular surfaces. In Section \ref{ReflectSec} we provide conditions so an infinitesimal reflection symmetry in a minimal Plateau surface propagates to a global symmetry.  Finally,  in Section \ref{section mp} we present the main moving planes argument.

\medskip

For future extensions to varifolds -- see \cite{VarifoldPaper} -- the results of this section will be proved for a more general class of surfaces than minimal Plateau surfaces.  Specifically, we consider a closed set, $\Sigma$, that is a minimal Plateau surface away from a discrete set, $Q$, of potentially exotic singularities.  We show that, under certain natural conditions on these singularities, neither they nor $T$-points occur in the region in which the moving planes method applies -- i.e., $\Sigma$ is a $Y$-surface in this region. More precisely, we require that, at the points of $Q$, $\Sigma$ has upper density strictly less than $2$. Here the {\bf upper density of $\Sigma$ at $p$} is defined to be
$$
\bar{\Theta}(\Sigma, p)=\limsup_{r\to 0^+} \frac{\mathcal{H}^2(\Sigma\cap B_r(p))}{\pi r^2}\,,
$$
When the usual limit exists, we denote it by $\Theta(\Sigma,p)$ and call it the {\bf density of $\Sigma$ at $p$}. If $\Sigma$ is a minimal Plateau surface in a neighborhood of $p$, then $\bar{\Theta}(\Sigma, p)<2$ as,
\begin{equation}
\label{density for minimal plat surf}
\Theta(\Sigma, p)=\left\{
\begin{split}
&1/2\,,&\quad\mbox{if $p\in \partial \Sigma$}\,,
\\
&1\,,&\quad\mbox{if $p\in \mathrm{int}(\Sigma)$}\,,
\\
&3/2\,,&\quad\mbox{if $p\in \Sigma_Y$}\,,
\\
&\frac{6}{2\pi} \arccos\Big(-1/3 \Big) \approx 1.82\,,&\quad\mbox{if $p\in \Sigma_T$}.
\end{split}
\right .
\end{equation}

\medskip

\subsection{Removable singularities and unique continuation}\label{section rs and uc}
We first prove a removable singularities result for minimal Plateau surfaces, see Lemma \ref{RemoveSingLem}. The starting point is the observation that, if $\Sigma$ is a minimal Plateau surface in $U$, then, for any $X\in C^1_c(U;\R^3)$,
\begin{equation}
\label{minimal PS}
\int_{\reg(\Sigma)}\,\Div^\Sigma X\,d\Ha^2=\int_{\pa\Sigma}X\cdot\,\nu^{{\rm co}}_{\Sigma}\,d\Ha^1\,.
\end{equation}
In particular, the multiplicity one rectifiable varifold $V_\Sigma$ defined by $\Sigma$ is stationary in $U\setminus \partial \Sigma$.

\medskip

To prove \eqref{minimal PS}, we notice that if $S$ is a connected component of ${\rm int}(\Sigma)$, then $\overline{S}$ is a surface with boundary in the open set $U\setminus\Sigma_T$, with ${\rm int}(\overline{S})=S$ and $\partial \overline{S}=[\overline{S}\cap\Sigma_Y]\cup[\overline{S}\cap\partial \Sigma]$. Moreover, at each $p\in \overline{S}\cap\Sigma_T$, $\overline{S}$ is locally diffeomorphic to a planar angular sector (isometric to one of the six angular sectors forming $T$), so that the classical proof of the tangential divergence theorem can be easily adapted to $\overline{S}$. We thus have that for every $X\in C^1_c(U;\R^3)$
\[
\int_{S}\,\Div^{S} X\,d\Ha^2=\int_{\partial \overline{S}}X\cdot\,\nu^{{\rm co}}_{\overline{S}}\,d\Ha^1
=\int_{\Sigma_Y\cap\partial \overline{S}}X\cdot\,\nu^{{\rm co}}_{\overline{S}}\,d\Ha^1
+\int_{\partial\Sigma\cap\partial \overline{S}}X\cdot\,\nu^{{\rm co}}_{\overline{S}}\,d\Ha^1\,.
\]
Since ${\rm int}(\Sigma)$ has locally in $U$ finitely many connected components, the family $\mathcal{S}(X)$ of those components $S$  of ${\rm int}(\Sigma)$ such that $S\cap\spt X\ne\emptyset$ is finite. We thus find
\begin{eqnarray*}
&&\sum_{S\in\mathcal{S}(X)}\int_{S}\,\Div^{S} X\,d\Ha^2= \int_{{\rm int}(\Sigma)}\,\Div^\Sigma X\,d\Ha^2=\int_{\reg(\Sigma)}\,\Div^\Sigma X\,d\Ha^2\,,
\\
&&\sum_{S\in\mathcal{S}(X)} \int_{\partial\Sigma\cap\partial \overline{S}}X\cdot\,\nu^{{\rm co}}_{\overline{S}}\,d\Ha^1=\int_{\pa\Sigma}X\cdot\,\nu^{{\rm co}}_{\Sigma}\,d\Ha^1\,,
\end{eqnarray*}
where in the first identity we have used that $\Div^SX=\Div^\Sigma X$ on each $S$, while in the second identity we have used the observation that for each $p\in\partial\Sigma$ there exists exactly one $S\in\mathcal{S}(X)$ such that $p\in\overline{S}\cap\partial\Sigma$ and $\nu^{{\rm co}}_{\overline{S}}(p)=\nu^{{\rm co}}_\Sigma(p)$. We are left to show that
\[
\sum_{S\in\mathcal{S}(X)}\int_{\Sigma_Y\cap \overline{S}}X\cdot\,\nu^{{\rm co}}_{\overline{S}}\,d\Ha^1=0\,.
\]
The reason is that for each $p\in\spt(X)\cap\Sigma_Y$ there are $r>0$ and three distinct $S_1, S_2,S_3\in\mathcal{S}(X)$ such that $\Sigma\cap B_r(p)=(S_1\cup S_2\cup S_3)\cap B_r(p)\subset\subset U$, $p\in S_1\cap S_2\cap S_3$ and
\[
\nu_{\overline{S}_1}^{{\rm co}}(p)+\nu_{\overline{S}_2}^{{\rm co}}(p)+\nu_{\overline{S}_3}^{{\rm co}}(p)=0\,,
\]
since $\Sigma\cap B_r(p)$ is diffeomorphic to $Y$ through a map whose differential is an isometry at $p$. In particular $\sum_{S\in\mathcal{S}(X)}\nu^{{\rm co}}_{\overline{S}}=0$ on $\Sigma_Y\cap\spt X$, and therefore
\[
\sum_{S\in\mathcal{S}(X)}\int_{\Sigma_Y\cap \overline{S}}X\cdot\,\nu^{{\rm co}}_{\overline{S}}\,d\Ha^1=
\int_{\Sigma_Y\cap\spt X}X\cdot\,\Big(\sum_{S\in\mathcal{S}(X)}\nu^{{\rm co}}_{\overline{S}}\Big)\,d\Ha^1=0\,.
\]
This proves \eqref{minimal PS}.

\begin{lem}[Removable singularities for minimal Plateau surfaces]\label{RemoveSingLem}
	Let $\Sigma$ be a closed subset of $B_R=B_R(0)$ without isolated points so that  $\Sigma\backslash \set{0}$ is a minimal Plateau surface without boundary in $B_{R}\backslash \set{0}$.  If $\bar{\Theta}(\Sigma, 0)<2$ and $\Sigma\cap \set{x_3\geq 0}$ is a graph of locally bounded slope, then $\Sigma$ is a minimal Plateau surface in $B_R$. If $0\in \Sigma$, then $0\in \mathrm{int}(\Sigma)\cup \Sigma_Y$; and if $0\in\Sigma_Y$, then the spine of $T_0 \Sigma$ lies on $\set{x_3=0}$.
\end{lem}

\begin{proof} Since $\Sigma$ is closed in $B_R$, if $0\not\in\Sigma$, then $B_r\cap\Sigma=\emptyset$ for some $r>0$, and thus the fact that $\Sigma$ is a minimal Plateau surface in $B_R\setminus\set{0}$ implies that $\Sigma$ is a minimal Plateau surface in $B_R$. We can thus assume that $0\in\Sigma$.
	
\medskip
	
\noindent {\it Step one}: We first prove that
\begin{equation}
	\label{sigma is stationary}
	\int_{\reg(\Sigma)}\Div^\Sigma X=0\qquad\forall X\in C^\infty_c(B_R;\R^3)\,,
\end{equation}
that is, the rectifiable varifold $V_\Sigma$ defined by $\Sigma$ is stationary in $B_R$. As \eqref{minimal PS} holds for $\Sigma$ in $B_{R}\backslash \set{0}$, we have $\int_{\reg(\Sigma)}\Div^\Sigma Y d\mathcal{H}^2=0$ for every $Y\in C^\infty_c(B_{R}\backslash \set{0};\R^3)$. Setting $Y=\eta_\varepsilon\,X$ for $X\in C^\infty_c(B_R;\R^3)$ and $\eta_\varepsilon$ a smooth cutoff with $\eta_\varepsilon=1$ on $\R^3\setminus B_\varepsilon$ and $\eta_\varepsilon=0$ on $B_{\varepsilon/2}$, we thus find
\[
\int_{\reg(\Sigma)}\,\eta_\varepsilon\,\Div^\Sigma X=-\int_{\reg(\Sigma)} X\cdot\nabla\eta_\varepsilon\,.
\]
Choosing $\eta_\varepsilon$ so that $\eta_\varepsilon\to 1$ on $\R^3\setminus\set{0}$ and  $|\nabla\eta_\varepsilon|\le 1_{B_\varepsilon\setminus B_{\varepsilon/2}}\,C/\varepsilon$, we have
\[
\Big|\int_{\reg(\Sigma)}\Div^\Sigma X\Big|\le C\,\|X\|_{C^0}\,\limsup_{\varepsilon\to 0^+}\frac{\mathcal{H}^2(\Sigma\cap B_\varepsilon)}{\varepsilon}=0\,,
\]
where we have used $\bar\Theta(\Sigma,0)<\infty$ to deduce $\mathcal{H}^2(\Sigma\cap B_\varepsilon)={\rm o}(\varepsilon)$ as $\varepsilon\to 0^+$. We have thus proved that \eqref{sigma is stationary} holds.

\medskip

\noindent {\it Step two}: We show that $\Theta(\Sigma,0)$ exists and belongs to $[1,2)$. By \eqref{sigma is stationary} and the monotonicity formula for stationary varifolds \cite[Section 17]{SimonLN}, $\Theta(\Sigma,p)$ exists at every $p\in B_R$ and defines an upper semicontinuous function on $B_R$. As $\Sigma$ contains no isolated points and $0\in \Sigma$, there are $p_j\to 0$ as $j\to\infty$ with $p_j\in\Sigma\setminus\{0\}$. By upper semicontinuity of $\Theta(\Sigma,\cdot)$ in $B_R$ we have
\[
\Theta(\Sigma,0)\ge\limsup_{j\to\infty}\Theta(\Sigma,p_j)\ge 1\,,
\]
where we have used \eqref{density for minimal plat surf} and the assumption that $\Sigma$ has no boundary points in $B_R\setminus\{0\}$ to obtain $\Theta(\Sigma,p_j)\ge 1$ for every $j$. Hence, $1\le\Theta(\Sigma,0)\leq \bar{\Theta}(\Sigma,0)<2$.

\medskip

\noindent {\it Step three}: We show that every varifold blow-up limit $\mathcal{C}$ of $V_\Sigma$ at $0$ has multiplicity one and satisfies
\begin{equation}
  \label{upper part of C is halfplanes}
  \mathcal{C}=V_K\,,\qquad
  \mathcal{C}\llcorner \set{x_3\geq 0}=\sum_{i=1}^N V_{H_i}
\end{equation}
where $K=\spt\mathcal{C}$ is a cone with vertex at $0$, $H_i\subset \set{x_3\geq 0}$ are half-planes with $\partial H_i=\ell\subset \set{x_3=0}$ for $1\leq i\leq N$, $\ell$ is a line in $\set{x_3=0}$, and where $V_K$ and $V_{H_i}$ are the multiplicity one varifolds associated to $K$ and $H_i$ respectively. Indeed,
given a sequence of radii $\rho_i\to 0^+$, up to extracting a subsequence, the multiplicity one varifolds $V_{\Sigma/\rho_i}$ have a varifold limit $\mathcal{C}$ which is an integer stationary varifold in $\R^3$ supported on a cone $K$ with vertex at $0$. If $\theta$ denotes the multiplicity of $\mathcal{C}$ and $q$ is a Lebesgue point of $\theta$ with $\theta\ge2$, then $2\le\Theta(\mathcal{C},q)=\Theta(\mathcal{C},t\,q)$ for every $t>0$, thus leading to $\Theta(\mathcal{C},0)\ge2$ by upper semicontinuity of $\Theta(\mathcal{C},\cdot)$ on $\R^3$. Since $\Theta(\mathcal{C},0)=\Theta(\Sigma,0)\in[1,2)$, we deduce that $\theta=1$ $\|\mathcal{C}\|$-a.e., and thus that $\mathcal{C}=V_K$. Concerning the second identity in \eqref{upper part of C is halfplanes}, we notice that since $\Sigma\cap \set{x_3\geq 0}$ is a graph of locally bounded slope and $\Sigma$ is a minimal Plateau surface without boundary in $B_R\setminus\{0\}$, it follows that $\Sigma \cap \set{x_3>0}$ is a smooth, stable minimal surface in $B_R\cap\{x_3>0\}$. (Notice that it is possible for $\Sigma\cap\set{x_3>0}$ to be empty!) Hence, for $q\in\Sigma\cap\{x_3>0\}\cap B_{R/2}$,  $\Sigma$ is a stable minimal surface in $B_{x_3(q)}(q)$, and thus, by the curvature estimates of Fischer-Colbrie and Schoen \cite{FCS},
\begin{equation}\label{CurvEstEqn}
|A_{\Sigma}(q)|\leq \frac{C}{x_3(q)}\,,\qquad \forall q\in {\Sigma\cap B_{R/2} \cap \set{x_3>0}}
\end{equation}
where $C>0$ is a universal constant. Since \eqref{CurvEstEqn} implies
\[
|A_{\Sigma/\rho_i}(q)|\leq \frac{C}{x_3(q)}\,,\qquad \forall q\in {\Sigma/\rho_i\cap B_{R/2\rho_i} \cap \set{x_3>0}}
\]
we deduce that $K\cap\{x_3>0\}$ is a smooth minimal surface. Since $K\cap\{x_3>0\}$ is a cone with respect to $0$ we deduce that $K\cap\{x_3>0\}$ is a finite union of half-spaces bounded by a same line $\ell$ contained in $\{x_3=0\}$. This also implies that $K\cap\{x_3=0\}$ is either equal to $\ell$, or to $\{x_3=0\}$, or to one of the two half-spaces bounded by $\ell$ in $\{x_3=0\}$. Since $\theta=1$ $\|\mathcal{C}\|$-a.e. on $K$ we complete the proof of \eqref{upper part of C is halfplanes}.

\medskip

\noindent {\it Step four}: We complete the proof. Since $\mathcal{C}=V_K$ is a stationary multiplicity one conical varifold, $K\cap\pa B_1$ induces a one-dimensional, multiplicity one stationary varifold on $\pa B_1$. A result of Allard and Almgren \cite{allardalmgren76} implies that, for every $p\in K\cap\pa B_1$, there is $r>0$ such that $B_r(p)\cap K\cap\pa B_1$ is a finite union of geodesic arcs originating from $p$ along directions $\{v_j(p)\}_{j=1}^{m(p)}\subset T_p(\pa B_1)$ such that $\sum_{j=1}^{m(p)}v_j(p)=0$. Thus, in the terminology of Appendix \ref{appendix geonets}, we find that
\begin{eqnarray}\label{gamma is a geodesic net}
  &&\mbox{$\Gamma=K\cap\pa B_1$ is a geodesic net in $\pa B_1$,}
  \\\nonumber
  &&\mbox{with $m(p)\in\{2,3\}$ for every $p\in\Gamma$}\,.
\end{eqnarray}
(Here $m(p)=\#\,I(p)$ for $I(p)$ as in the appendix). Indeed,  $m(p)\in\{2,3\}$ follows immediately from the upper semicontinuity of $\Theta(\mathcal{C},\cdot)$, from $\Theta(\mathcal{C},0)\in[1,2)$, and
\[
\Theta(\mathcal{C},t\,p)=\Theta(\mathcal{C},p)\ge\frac{m(p)}2\,,\qquad\forall t>0\,,\forall p\in\Gamma\,.
\]
If now $K\cap\{x_3<0\}=\emptyset$, then $0\in\Sigma$ implies $K=\{x_3=0\}$ by a standard first variation argument (see, e.g., \cite[Lemma 3]{delgadinomaggiAPDE}), and thus $\Sigma$ is a smooth minimal surface in a neighborhood of $0$ by Allard's regularity theorem \cite{Allard}. We can therefore assume that $K\cap\{x_3<0\}\ne\emptyset$, and thus, since $K$ is a cone, that
\[
\exists\, p_1\in\Gamma\cap\{x_3<0\}\,.
\]
We claim that the existence of $p_1$ ensures that, if $\{p_0,-p_0\}=\ell\cap\pa B_1\subset\Gamma$, then
\begin{eqnarray}\label{p0 has vector}
  \mbox{$v_j(p_0)\cdot e_3>0$ for some $j=1,...,m(p_0)$}\,.
\end{eqnarray}
Indeed, $\Gamma\cap\{x_3\ge0\}$ consists of equatorial half-circles with $\pm p_0$ as end-points; at the same time $\Gamma$ is connected (as a consequence of the stationarity of $\mathcal{C}=V_K$); therefore the only way to connect a point in $\Gamma\cap\{x_3<0\}$ to, say, $p_0$ is the existence of a geodesic arc whose interior is entirely contained in $\{x_3<0\}$ and having one endpoint at either $p_0$ or $-p_0$. Therefore up to exchange the roles of $p_0$ and $-p_0$ we can assert \eqref{p0 has vector}. (Notice that \eqref{p0 has vector} must hold at both endpoints of $\ell\cap\pa B_1$, but we shall not need this fact). By \eqref{upper part of C is halfplanes} and \eqref{p0 has vector} we deduce that
\begin{equation}
  \label{necessarily}
  m(p_0)=N+\#\Big\{j:v_j(p_0)\cdot e_3>0\Big\}>N\,.
\end{equation}
A first consequence of \eqref{necessarily} is that $N\ge 3$ would imply $m(p_0)\ge 4$, contradicting \eqref{gamma is a geodesic net}: hence $N\in\{1,2\}$. If $N=1$ and $m(p_0)=m(-p_0)=2$, then $\Gamma$ is a geodesic net which agrees with an equatorial circle on $\{x_3>-\varepsilon\}$ for some $\varepsilon>0$: therefore $\Gamma$ is an equatorial circle by Lemma \ref{lemma geodesic nets} in Appendix \ref{appendix geonets}, $K$ is a plane, and $\Sigma$ is a smooth minimal surface near $0$. If $N=1$ and $m(p_0)=m(-p_0)=3$, then $\Gamma$ is a geodesic net which agrees with a $Y$-net on $\{x_3>-\varepsilon\}$ for some $\varepsilon>0$: therefore $\Gamma$ is a $Y$-net by Lemma \ref{lemma geodesic nets}, $\Theta(\mathcal{C},0)=3/2$ and by \cite[Corollary 3 in Section 1, Remark 2 in Section 7]{SimonCylindrical} $\Sigma$ is a  $Y$-surface in a neighborhood of $0$. If $N=2$, then, again by \eqref{necessarily},  $m(p_0)=m(-p_0)=1$ and so  $\Gamma$ agrees with a $Y$-net on $\{x_3>-\varepsilon\}$ for some $\varepsilon>0$, and we conclude as before. The remaining situation is $N=1$, $m(p_0)=3$ and $m(-p_0)=2$. In this case, by slightly tilting the vector $e_3$ into a new unit vector $e_3^*$, see
\begin{figure}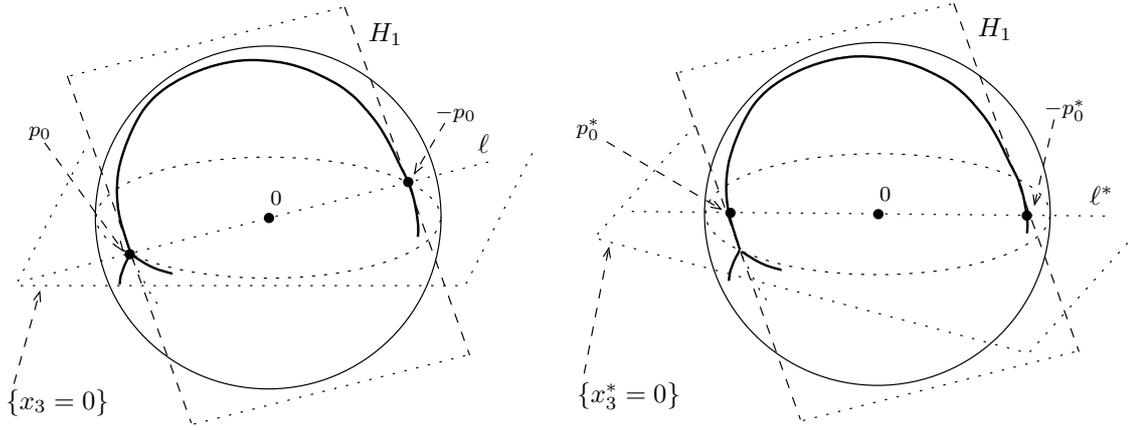\caption{\small{Step three of the proof of Lemma \ref{RemoveSingLem}, the case when $N=1$ (i.e., $K\cap\{x_3\ge0\}=H_1$ with $\ell=H_1\cap\{x_3=0\}$), $m(p_0)=3$ and $m(-p_0)=2$, where $p_0$ and $-p_0$ are the endpoints of $\ell\cap B_1$. The fact that $m(-p_0)=2$ implies that $K\cap\pa B_1$ coincides with $H_1\cap\pa B_1$ near $-p_0$, while $m(p_0)=3$ means $p_0$ is a $Y$-point and one of three arcs emanating from $p_0$ is given by $H\cap\pa B_1\cap\{x_3\ge0\}$. In this situation we tilt $e_3$ into a new vector $e_3^*$ and shrink $\varepsilon>0$ so  $\Gamma\cap\{x_3^*>-\varepsilon\}=H_1\cap\{x_3^*>-\varepsilon\}$. Clearly, if $\ell^*=H_1\cap\{x_3^*=0\}$ and $\pm p_0^*$ are the endpoints of $\ell^*\cap\pa B_1$, then $m(\pm p_0^*)=2$.}}\label{fig p0}\end{figure}
Figure \ref{fig p0}, we find that, for some $\varepsilon>0$, $\Gamma\cap\{x_3^*>-\varepsilon\}$ is a great circle: hence $\Gamma$ must be an equatorial circle by Lemma \ref{lemma geodesic nets}, contradicting $m(p_0)=3$.
\end{proof}

We next prove a kind of unique continuation result for minimal Plateau surfaces lying on one side of a regular minimal surface.  This is slightly subtle as the usual unique continuation principle does not directly hold for minimal Plateau surfaces.

\begin{lem}[Unique continuation for minimal Plateau surfaces]\label{UniqueContLem}
	Let $U\subset\R^3$ be open, $Q=\set{q_1, \ldots, q_N}$ be a finite set of points in $U$, $\Sigma_1$ be a connected, (relatively) closed set in $U$ and assume that $\Sigma_1\backslash Q$ is a minimal Plateau surface without boundary in $U\backslash Q$ with $\bar{\Theta}(\Sigma_1, q)<2$ for each $q\in Q$.
	Suppose there is an open subset $V\subset U$ so that $\Sigma_2\subset U\cap\partial V$ is a regular minimal surface without boundary. If $V\cap \Sigma_1=\emptyset$ and there is a point $p_0\in \Sigma_1\cap \Sigma_2$, then $\Sigma_1\subset \Sigma_2$. If $\Sigma_2$ is connected, then $\Sigma_1=\Sigma_2$.
\end{lem}

\begin{proof} By throwing out points of $Q$ if needed, we may assume $\Sigma_1$ is not a regular minimal surface in a neighborhood of any point of $Q$. Set $\Gamma=\Sigma_1\cap \Sigma_2$.  We claim
\begin{equation}
\label{uc1}
\left( Q\cup  \sing(\Sigma_1\backslash Q)\right)\cap \Gamma= \emptyset.
\end{equation}
Indeed, as in step one of the proof of Lemma \ref{RemoveSingLem}, the multiplicity one varifold $V_{\Sigma_1}$ defined by $\Sigma_1$ is stationary in $U$. If $q\in \Gamma$, then as $\Sigma_2$ is smooth and $\Sigma_2\subset\partial V$, there is an open half-space $H$ so $T_q\Sigma_2=P=\partial H=T_q V$.  As $\Sigma_1\cap V=\emptyset$, any tangent cone, $\mathcal{C}$,  to $V_{\Sigma_1}$ at $q$ has support disjoint from $H$ and, because $\Sigma_1$ is connected,  $\Theta(\Sigma, q)\geq 1$ and so $\mathcal{C}$ is non-trivial. As $\mathcal{C}$ is a stationary integer multiplicity cone with density strictly less than $2$, this implies that $\mathcal{C}=V_P$,  Hence, by Allard's theorem \cite{Allard}, $q$ is a regular point of $V_{\Sigma_1}$, and so $q\not \in Q\cup \sing(\Sigma_1\backslash Q)$.

In particular, by \eqref{uc1}, $\Sigma_2\cap\mathrm{int}(\Sigma_1\backslash Q)=\Gamma$.  Hence, for any $q\in \Gamma$ there is an $r>0$ so $\Sigma_1'=B_r(q)\cap \Sigma_1$ and $\Sigma_2'=B_{r}(q)\cap \Sigma_2$ are connected regular minimal surfaces with $\Sigma_1'$ lying on one side of $\Sigma_2'$ and $q\in \Sigma_1'\cap \Sigma_2'$.  The strong maximum principle immediately implies $\Sigma_1'=\Sigma_2'\subset \Gamma$, i.e.,  $\Gamma$ is an open subset of $\Sigma_1$.  As $\Gamma$ is also clearly a closed subset of $\Sigma_1$ and $p_0\in \Gamma$, the connectedness of $\Sigma_1$ implies  $\Sigma_1=\Gamma\subset \Sigma_2$.  Likewise, $\Sigma_1=\Gamma$ is an open and closed subset of $\Sigma_2$, proving the last claim.
	\end{proof}

Finally, we observe that there is an injective map from the cells of the tangent cone at a non-boundary point of a cellular minimal Plateau surface to its own cells.
\begin{lem}\label{CellularLem}
Let $U\subset \Real^3$ be open and $\Sigma\subset U$ be minimal Plateau surface without boundary in $U$ that is cellular in $U$.  For each $p\in \Sigma$, $T_p\Sigma$ is cellular in $\Real^3$.  Moreover, there is a well defined injective map
$$
\mathcal{I}_p : \mathcal{C}(T_p\Sigma)=\set{W^j}_{j=1}^M\to \mathcal{C}(\Sigma)=\{U^i\}_{i=1}^N
$$
defined by $\mathcal{I}_p(W^j)=U^{i_j}$ when and only when
$$W^j=\lim_{\rho\to 0} \rho^{-1}(U^{i_j}-p),$$
where the convergence occurs in $L^1_{{\rm loc}}(\R^3)$ for the corresponding indicator functions.
\end{lem}
\begin{proof}
 By inspection, $P, Y$ and $T$ are cellular in $\Real^3$ with two,  three, and four cells respectively.  Hence, each $T_p\Sigma$ is cellular in $\Real^3$.
  By the definition of minimal Plateau surface, there exist an $r>0$ and a $C^{1,\alpha}$-diffeomorphism $\phi:B_r(p)\to \Real^3$ such that $\phi(p)=0$, $D\phi_p=I$, the identity map and $\phi(B_r(p)\cap\Sigma)= {T}_p\Sigma$.  In particular, there is a $0<r_1<r$ so for $0<r'<r_1$, $B_{\frac{1}{2}r'}(0)\subset \phi(B_{r'}(p))\subset B_{2r'}(0)$.  Moreover, one has $\mathcal{I}_p(W^j)=U^{i_j}$ if and only if, for any $0<r'<r$, $W^j\cap B_{\frac{1}{2}r'}(0)\subset \phi(U^{i_j}\cap B_{r'}(p))$.

  It remains only to show $\mathcal{I}_p$ is injective. As $\Sigma$ defines a cell structure in $U$, there is a $\rho>0$ so that $B_{\rho}(p)\subset U$ and, for every $0<\rho'<\rho$,  $B_{\rho'}(p)\cap U^{i}$ is connected.  Let $r_2=\frac{1}{2}\min \set{r_1, \rho}$.  Now suppose $\mathcal{I}_p(W^{j})=U^{i_j}=\mathcal{I}_p(W^{k})$.  As observed, this means $(W^j\cup W^k)\cap B_{\frac{1}{2} r_2} (0) \subset  \phi(U^{i_j}\cap B_{r_2}(p))$.   As $U^{i_j}\cap B_{r_2}(p)$ is connected, so is $\phi(U^{i_j}\cap B_{r_2}(p))\subset \Real^3\backslash T_p\Sigma$ and so it must be that $W^j=W^k$ and so $\mathcal{I}_p$ is injective.
 \end{proof}

\subsection{Reflection symmetry}\label{ReflectSec}
An important technical consequence of the unique continuation result, Lemma \ref{UniqueContLem}, and the Hopf maximum principle is that, under suitably hypotheses, an infinitesimal symmetry of a minimal Plateau surface (assumption (4) in Lemma \ref{SymmetryLem} below) extends to a global symmetry (the conclusion $R_0(\Sigma^+) \subset \Sigma$ in the same lemma).
 In order to state this precisely, it is helpful to recall some additional notation from  \cite{SchoenSymmetry}.
First let $R_t$ denote the reflection map through $\{x_3=t\}$, so that
\[
R_t(\mathbf{y}, x_3)=(\mathbf{y}, 2\,t-x_3)\,.
\]
If $\Sigma \subset \Real^3$ and $t\in \Real$, we let
$$
\Sigma_{t^+}=\Sigma\cap\{x_3\ge t\} \mbox{ and } \Sigma_{t^+}^\circ= \Sigma \cap \set{x_3>t}.
$$
Similarly, let
$$
\Sigma_{t^-}=\Sigma\cap\{x_3\le t\}\mbox{ and } \Sigma_{t^-}^\circ= \Sigma \cap \set{x_3<t}.
$$
Observe that, due to the possible presence of a floating disk in $\set{x_3=t}$, one may have $\Sigma_{t^\pm }^\circ\subsetneq \bar{\Sigma}_{t^\pm }^\circ\subsetneq \Sigma_{t^\pm}$, where $\bar{\Sigma}_{t^+}^\circ$ is the closure of ${\Sigma}_{t^+}^\circ$.


\begin{lem}\label{SymmetryLem}
	Let $U$ be an open set so that $R_0(U)\subset U$ and let $Q\subset U\cap \set{x_3<0}$ be a finite set of points. Suppose $\Sigma\subset U$ is a closed set so $\Sigma\backslash Q$ is a minimal Plateau surface without boundary in $U$ and $\bar{\Theta}(\Sigma,q)<2$ for all $q\in Q$.  If $\Sigma^+$ is a component of ${\Sigma}^\circ_{0^+}$ and $V$ is an open subset of $U\cap\set{x_3>0}$ so that:
	\begin{enumerate}
		\item $\Sigma^+$ is a connected regular minimal surface in $\set{x_3>0}$;
		\item $ \Sigma^+\subset \partial V$ in $U\cap \set{x_3>0}$;
		\item  $R_0(V)\cap \Sigma=\emptyset$;
		\item There is a $p\in \set{x_3=0}\cap \bar{\Sigma}^+$ so that $R_0(T_p\Sigma)=T_p\Sigma$ and $T_p\Sigma \neq \set{x_3=0}$,
	\end{enumerate}
	then $R_0(\Sigma^+) \subset \Sigma$.
\end{lem}
\begin{proof}
	First observe that if $\Sigma$ is regular at $p$, then (4) implies that $T_p\Sigma$ is a plane orthogonal to $\set{x_3=0}$, i.e., a vertical plane.  Likewise, if $p$ is a singular point, then $T_p\Sigma$ is a $Y$ whose spine, $\ell$, lies in $\set{x_3=0}$ and so that one of the half-planes making up $Y\backslash \ell$ is contained in $\set{x_3=0}$. Hence, up to rotating around the $x_3$-axis, which leaves all hypotheses unchanged, we may assume $T_p\Sigma$ is $\set{x_1=0}$ in the regular case, or $T_p\Sigma=H_0\cup H_{120}\cup H_{-120}$ in the singular case. We first prove that $\Sigma$ is locally symmetric near $p$. That is, there is a $R>0$ so that $R_0(\Sigma^+)\cap B_R(p)\subset \Sigma$.
	
	\medskip
	
	\noindent {\it Local symmetry in the regular case}:  As $T_p\Sigma=\set{x_1=0}$,  there is a radius $r>0$ so that $B_{2r}(p)\cap \Sigma$ is a smooth surface and there is a solution to the minimal surface equation $u: D_r=\set{(0,s,t): s^2+t^2<r^2}\to \Real$ so that $u(0)=0$, $\nabla u(0)=0$ and
	$$
	\Sigma\cap B_{r/2}(p)\subset \set{(x_1(p)+u(s,t), x_2(p)+s, x_3(p)+t): (s,t)\in D_r}\subset \Sigma\cap B_{2r}(p).
    $$
    Let $V_-=D_r\cap \set{x_3\leq 0}$ be the closed half-disk. Set $v_-=u|_{V_-}$ and ${v}_+=(u\circ R_0)|_{V_-}$.  Clearly, $v_\pm$ satisfy the minimal surface equation on $V_-$,  $v_\pm(0)=0$ and $\nabla v_\pm(0)=0$.  Up to rotation around the $x_3$-axis by $180^\circ$, condition (2) and (3) imply that $v_+\geq v_-$ on $V_-$.  In particular, up to shrinking $r$, $w=v_+-v_-$ is a non-negative solution to a uniformly elliptic equation on $V_-$ with $w(0)=0$ and $\nabla w(0)=0$ and so, by the  Hopf maximum principle, $v\equiv 0$. That is $v_-\equiv v_+$ on $V_-$ and so claim holds with $R=r/2$.

    \medskip
	
	\noindent {\it Local symmetry in the singular case}: In this case $T_p \Sigma=H_0\cup H_{120}\cup H_{-120}$ and there exist $r>0$, so that
	$\Sigma\cap B_{2r}(p)$ is a $Y$-surface.  Indeed, by taking $r$ small enough there are two $C^{1,\alpha}$-domains with boundary $V_\pm\subset D_r=\{(0,s,t):s^2+t^2<r\}\subset\{x_1=0\}$ so that $D_r=V_+\cup V_-$ and
	\[
	\big\{(0,0, \pm t); t\in (0, r)\big\}\subset V_\pm\,,\qquad \mbox{$\eta=D_r\cap\partial V_+\cap\partial V_-$ is a $C^{1,\alpha}$ curve}\,,
	\]
	and two smooth solutions to the minimal surface equation $u_\pm:V_\pm\to \Real$ so
	\begin{equation}\label{upm}
	u_\pm(0)=0\,,\qquad u_+|_\eta=u_-|_{\eta}\,,\qquad \nabla u_{\pm}(0)=(0, \mp\sqrt{3})
	\end{equation}
 and
	\[
	\Sigma\cap B_{r/2}(p)\subset \Big\{(x_1(p)+u_\pm (s,t),x_2(p)+s, x_3(p)+t): (s,t)\in V_\pm \Big\}\subset \Sigma\,;
	\]
	see
	\begin{figure}
		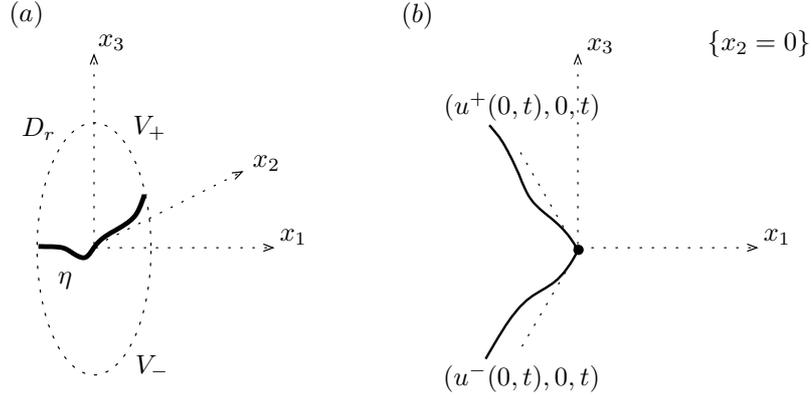\caption{{\small A visualization of the case $T_p \Sigma=H_0\cup H_{120}\cup H_{-120}$. In a neighborhood of $p$, $\Sigma$ contains two minimal graphs {\bf in the $x_1$-direction}, defined over complementary subdomains $V^\pm$ of a disk. The graphs meet along a $C^{1,\alpha}$-curve, and form a $120^\circ$ angle at $p$: (a) the domains of the  two graphs which are subsets of a disk $D_r\subset\{x_1=0\}$ centered at $0$; (b) a cross section by $\{x_2=0\}$ stresses the angle condition.}}\label{fig uplusmin}
	\end{figure}
	Figure \ref{fig uplusmin}. By hypothesis (1), $\Sigma$ is regular in $\set{x_3>0}$ and so $V_-\subset R_{0}(V_+)$ and so $v_+=(u_+\circ R_0)|_{V_-}$ is defined on the same domain, $V_-$, as $v_-=u_-$.

	 Clearly,  (2) and (3) imply either that $v_+\geq v_-$ on $V_-$ or $v_-\geq v_+$.  Indeed, the former occurs if
	 $$
	  \Big\{(x_1(p)+z,x_2(p)+s, x_3(p)+t): (s,t)\in V_+,  u_+(s,t)<z \Big\}\cap B_{r}(p)\subset V
	 $$
	 and the later occurs when
	$$
	\Big\{(x_1(p)+z,x_2(p)+s, x_3(p)+t): (s,t)\in V_+,  u_+(s,t)>z \Big\}\cap B_{r}(p)\subset V.
	$$
	We assume $v_+\geq v_-$, the proof is the same in the other case. Observe \eqref{upm} implies $v_+(0)=v_-(0)=0$ and $\nabla v_+(0)=\nabla v_-(0)$. As $v_-$ and $v_+$ both satisfy the minimal surface equation on $V_-$, up to shrinking $r$,  $w=v_+-v_-\geq 0$ satisfies a uniformly elliptic equation on $V_-$. As $w(0)=0$ and $\nabla w(0)=0$, the Hopf maximum principle for $C^{1, \alpha}$ domains -- see  \cite{LeoRosales} -- implies $w\equiv 0$, that is, $u_+\circ R_0=u_-$ on $V_-$. Hence, the claim holds with $R= r/2$.
	
	\medskip
	
	\noindent {\it Propagating the symmetry}: Finally, we apply Lemma \ref{UniqueContLem} to propagate the inclusion $R_0(\Sigma_+)\cap B_R(p)\subset \Sigma$ to  $R_0(\Sigma^+)\subset \Sigma$.
	Let $\Sigma_1$ be the component of ${\Sigma}_{0^-}^\circ$ whose closure contains $p$ -- such a component exists and is unique as $T_p\Sigma\cap \set{x_3<0}$ is connected and non-empty.  Set $\Sigma_2=R_0(\Sigma^+\cap \set{x_3>0})$, so that, in $U'=U\cap \set{x_3<0}$,
	$$
   \Sigma_2=R_0(\Sigma^+)\subset \partial\big( R_0(V)\big)
	$$
	and by hypothesis (3), $R_0(V)\cap \Sigma_1=\emptyset$. As $B_R(p)\cap \Sigma_1=B_R(p)\cap \Sigma_2$ and both $\Sigma_1$ and $\Sigma_2$ are connected, Lemma \ref{UniqueContLem} implies $R_0(\Sigma^+)=\Sigma_2=\Sigma_1\subset \Sigma$.
\end{proof}

\subsection{The moving planes argument}\label{section mp}
We now prove the  key technical result of the paper: Let  $\Sigma$ be a minimal Plateau surface in a convex cylinder $C_\Omega$ whose boundary $B$ is contained in the boundary of the cylinder. If $B_{0^+}$ is a graph of locally bounded slope and $B$ is ``ordered by reflection with respect to the plane $\{x_3=0\}$'' (assumption (b) below), then the same holds for $\Sigma$, i.e.,  $\Sigma_{0^+}$ is a graph of locally bounded slope, and $\Sigma$ is ordered by reflection, see conclusions (i) and (ii) of Theorem \ref{CylThm}.

In order to state this result concisely we recall the following partial order from \cite{SchoenSymmetry}. For subsets $A,B\subset \Real^3$ we write
\[
A\leq B
\]
if $\pi(A)= \pi(B)$, and if $(\mathbf{y},t)\in \pi^{-1}(\mathbf{y})\cap A$ and $(\mathbf{y}, t')\in \pi^{-1}(\mathbf{y})\cap B$ implies $t\le t'$. Here, as in the previous section, $\pi(\mathbf{y},t)=\mathbf{y}$ for every $(\mathbf{y},t)\in\Real^3$.

\begin{thm}\label{CylThm}
	Let $\Omega\subset \Real^2$ be a bounded, open convex set with $C^1$-boundary, and let the open cylinder over $\Omega$ be denoted by
	\[
    C_{\Omega}=\set{(\mathbf{y},x_3): \mathbf{y}\in \Omega}\,.
    \]
    Let $\Sigma\subset \Real^3$ be a compact set without isolated points and let $B\subset \partial C_{\Omega}$ be a closed, non-empty, one-dimensional $C^1$-submanifold (not necessarily connected).  Suppose that $B$ and $\Sigma$ satisfy the following:
	\begin{enumerate}
		\item[(a)] $B_{0^+}$ is a graph of locally bounded slope and $T_pB$ is not vertical for any $p\in B\cap\{x_3>0\}$;
		\item[(b)] $B_{0^-}\le R_0(B_{0^+})$;
		\item[(c)] $(\partial C_{\Omega})_{0^+}\backslash B_{0^+}$ has two connected components, denoted by $V^0$ and $V^1$;
		\item[(d)] $\Sigma\backslash Q$ is a minimal Plateau surface in $\R^3\setminus Q$, where $Q=\set{q_1, \ldots, q_M}$ is a finite subset of $C_{\Omega}$ and,  for every $i$, $\bar{\Theta}(\Sigma, q_i)<2$;
        \item[(e)] $\partial (\Sigma\backslash Q)=B$ and $\Sigma\backslash B\subset C_{\Omega}$;
		\item[(f)] $\Sigma\setminus Q$ defines a cell structure $\{U^i\}_{i=0}^N$ in $C_{\Omega}\backslash Q$, and for $i=0,1$ we have
        \[
        \bar{V}^i=\partial U^i\cap  ( \partial C_{\Omega})_{0^+}\,;
        \]
	\end{enumerate}
see
\begin{figure}
  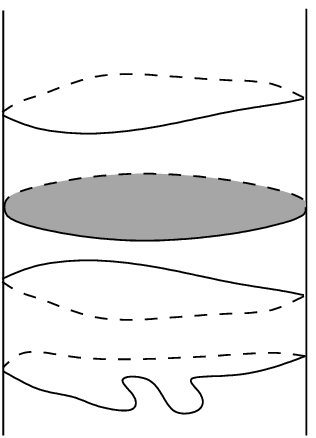\caption{{\small The situation in Theorem \ref{CylThm}. We consider a $C^1$-boundary data $B$ contained in $\pa C_\Om=(\pa\Om)\times\R$. The ``upper part'' $B_{0^+}$ of $B$ is a graph with bounded slope over $\pa\Om$, and, after reflection by $\{x_3=0\}$, it lies above the ``lower part'' $B_{0^-}$ of $B$ (which is not required to be a graph). The upper part of $\pa C_\Om$ is divided by $B_{0^+}$ into two components $V^0$ and $V^1$. We consider a  minimal Plateau surface with boundary $\Sigma$. If $\partial\Sigma$ is bounded by $B$ in such a way that $V^0$ and $V^1$ corresponds to the boundaries of the cells $U^0$ and $U^1$  defined by $\Sigma$ in $C_\Om$, then the theorem ensures that properties (a) and (b) of $B$ are ``transferred'' to $\Sigma$, see (i) and (ii).
  }}\label{fig cyl}
\end{figure}
Figure \ref{fig cyl}. Then
\begin{enumerate}
  \item[(i)] $\Sigma_{0^+}$ is a graph with locally bounded slope;
  \item[(ii)] $\Sigma_{0^-}\le R_0(\Sigma_{0^+})$;
  \item[(iii)] there is $\epsilon>0$ so that $\Sigma \cap \set{x_3>-\epsilon}$ is a minimal Plateau surface in $\set{x_3>-\epsilon}$.
\end{enumerate}
\end{thm}

Theorem \ref{CylThm}, whose proof is presented below, has the following corollary:
\begin{cor}\label{CylCor} Let $\Omega$, $B$, and $\Sigma$ satisfy the assumptions in Theorem \ref{CylThm}, but replace assumptions (b), (c) and (f) with
\begin{enumerate}
		\item[(b')] $B_{0^-}= R_0(B_{0^+})$;
		\item[(c')] $(\partial C_{\Omega})_{0^\pm}\backslash B_{0^\pm}$ has two connected components, denoted by $V^{0,\pm}$ and $V^{1,\pm}$;
		\item[(f')] $\Sigma\setminus Q$ defines a cell structure $\mathcal{C}=\{U^i\}_{i=0}^N$ in $C_{\Omega}\backslash Q$ and there are cells $U^{i,\pm}\in \mathcal{C}$ so that
\[
        \bar{V}^{i,\pm}=\partial U^{i,\pm}\cap  ( \partial C_{\Omega})_{0^\pm}\,,\qquad i=0,1\,,
        \]
       here $U^{i,+}$ and $U^{i,-}$ are not necessarily distinct elements of $\mathcal{C}$.
	\end{enumerate}
Then $R_0(\Sigma_{0^+})=\Sigma_{0^-}$ and $\Sigma$ is a minimal Plateau bi-graph.
\end{cor}

\begin{proof} Thanks to assumptions (b'), (c') and (f'), we can  apply Theorem \ref{CylThm} to both $\Sigma$ and $R_0(\Sigma)$, and so,  by conclusion (i), it is true that $\Sigma_{0^+}$ and $(R_0(\Sigma))_{0^+}=R_0(\Sigma_{0^-})$ are graphs of locally bounded slope.  Furthermore, conclusion (ii) implies
\[
\Sigma_{0^-}\le R_0(\Sigma_{0^+}) \mbox{ and } R_0(\Sigma_{0^+})=(R_0(\Sigma))_{0^-}\le R_0\big((R_0(\Sigma))_{0^+}\big)=\Sigma_{0^-}\,.
\]
Hence,  $R_0({\Sigma}_{0^+})=\Sigma_{0^-}$ and so $\Sigma=R_0(\Sigma)$. By conclusion (iii) of Theorem \ref{CylThm},  $\Sigma$ and $R_0(\Sigma)$ are both minimal Plateau surfaces in $\set{x_3>-\epsilon}$ for some $\epsilon>0$, and so $\Sigma$ is a minimal Plateau surface in $\R^3$. Finally, as $\Sigma_{0^-}=R_0({\Sigma}_{0^+})$ and $\Sigma_{0^+}$ are both graphs of locally bounded slope and $\sing(\Sigma)\subset \set{x_3=0}$, $\Sigma$ is a minimal Plateau bi-graph.
\end{proof}

\begin{proof}[Proof of Theorem \ref{CylThm}]
First observe that, by deleting points from $Q$, we may assume that $\Sigma$ is not a minimal Plateau surface in a neighborhood of any $q\in Q$. That is, the points of $Q$ are essential singularities of $\Sigma$. We define
\[
\sing(\Sigma)=\sing(\Sigma\backslash Q)\cup Q\,,
\]
where $\sing(\Sigma\backslash Q)$ is the singular set of $\Sigma\setminus Q$ as a minimal Plateau surface in $C_\Omega\setminus Q$. Similarly, let
\[
\reg(\Sigma)=\Sigma\backslash \sing(\Sigma)\,,\qquad \mathrm{int}(\Sigma)=\Sigma\backslash (B\cup \sing(\Sigma))\,.
\]
	
\noindent {\it Step one}: We establish some elementary facts.  First, we claim,
\[
T_+=\max\set{x_3(p):p\in B}>0> T_{-}=\min\set{x_3(p): p\in B}.
\]
Indeed, as $B$ is non-empty, either $B_{0^+}$ or $B_{0^-}$ is non-empty. Furthermore, $R_0(B_{0^+})\geq B_{0^-}$ requires that $\pi(B_{0^+})=\pi(B_{0^-})$ and so {\it both} $B_{0^+}$ and $B_{0^-}$ are non-empty. In particular, both $T_+$ and $T_-$ are finite. Clearly, $T_+\ge0$. If $T_+=0$, then $(\partial C_{\Omega})_{0^+}\backslash B_{0^+}$ has one connected component in $(\partial C_{\Omega})_{0^+}$, contradicting assumption (c) and so $T_+>0$, and because assumption (b) implies $T_-\le-T_+$ we conclude that $T_-<0$.

 Secondly, by the same argument used in step one of the proof of Lemma \ref{RemoveSingLem}, the multiplicity one varifold $V_\Sigma$ defined by $\Sigma$ is stationary in $\Real^3\backslash B$. Hence, the convex hull property of $V_\Sigma$ and the properties of $B$ imply
\begin{equation}
  \label{convex hull principle}
  \Sigma\subset\overline{C}_\Omega\cap\{T_-\leq x_3\le T_+\}\,.
\end{equation}

Finally, we review assumption (f): $\Sigma\setminus Q$ defines a cell structure $\mathcal{C}(\Sigma)=\{U^i\}_{i=0}^N$ in $C_{\Omega}\backslash Q$, so that the sets $U^i$ are open and connected, with
\begin{equation}
\label{cell structure2}
\partial(\Sigma\setminus Q) \subset \partial(C_{\Omega}\setminus Q)\,,\qquad C_{\Omega}\backslash (\Sigma\cup Q)=(C_\Omega\backslash Q)\backslash (\Sigma\backslash Q) =\bigcup_{i=0}^N U^i\,,
\end{equation}
and $\partial U^i\cap \partial C_{\Omega}=\bar{V}^i$ for $i=0,1$. As $\Sigma$ is compact and $C_{\Omega}$ has two components at infinity, corresponding to $x_{3}\to \pm \infty$ there is exactly one unbounded component of $\mathcal{C}(\Sigma)$ that contains points $p$ with $x_3(p)>T_+$. Up to a swapping $V^0$ and $V^1$, we may assume $U^1$ is this component and so $V^1$ is unbounded.

\medskip

\noindent {\it Step two}: We verify that the theorem holds in the ``trivial case'' where
\begin{equation}
  \label{trivial case}
  \mbox{$\overline{\Omega}\times\{T_+\}$ is a connected component of $\Sigma$}\,.
\end{equation}
Indeed, if this occurs, than the definition of minimal Plateau surface implies that there is a $\delta>0$ so that $\Sigma=\overline{\Omega}\times\{T_+\}$ in the slab $\{T_+-\delta<x_3<T_++\delta\}$.  In particular, $\partial\Sigma=\pa\Om\times\{T_+\}$ in this slab. As $\pa \Om \times \{T_+\}$ is a graph, assumption (a) implies $B_{0^+}=\pa\Om\times\{T_+\}$.  Hence, as $\Sigma$ cannot have a connected component without boundary points (indeed, by the convex hull principle every such component, being contained in the convex envelope of its boundary points, would be empty; see \cite[Theorem 19.2]{SimonLN}), we conclude that
\begin{equation}
  \label{trivial case stronger}
  \Sigma_{0^+}=\overline{\Omega}\times\{T_+\}\,.
\end{equation}
Conclusion (i) is thus immediate. By assumption (e) we have
\[
(\partial\Om)\times\{T_+\}=(\partial\Sigma)\cap\{x_3>0\}=B\cap\{x_3>0\}=B_{0^+}\cap\{x_3>0\}
\]
so that assumption (b) gives $B_{0^-}\le(\partial\Omega)\times\{-T_+\}$.  In particular,  $\Sigma_{0^-}$ is a minimal Plateau surface without boundary in $\set{x_3>-T_+}\backslash Q$ and so the varifold $V_{\Sigma_{0^-}}$ defined by $\Sigma_{0^-}$ is stationary in $\Real^3\backslash B_{0^-}$. Hence, the convex hull property implies,
\[
\Sigma_{0^-}\subset\overline\Omega\times(-\infty,-T_+]\,,
\]
which implies conclusion (ii). Finally, by (ii) and \eqref{trivial case stronger} it follows that $\Sigma\cap\{x_3>-T_+\}=\overline{\Omega}\times\{T_+\}$, so that conclusion (iii) holds. Having proved the theorem when \eqref{trivial case} holds, we will henceforth assume that {\it \eqref{trivial case} does not hold}.

\medskip

\noindent {\it Step three}: Begin the moving planes argument. For $t\in (0, T_+)$ and $U^i\in\mathcal{C}(\Sigma)$, let
	\[
	A^i=U^i\cap\{x_3<T_+\}
	\]
see
\begin{figure}
  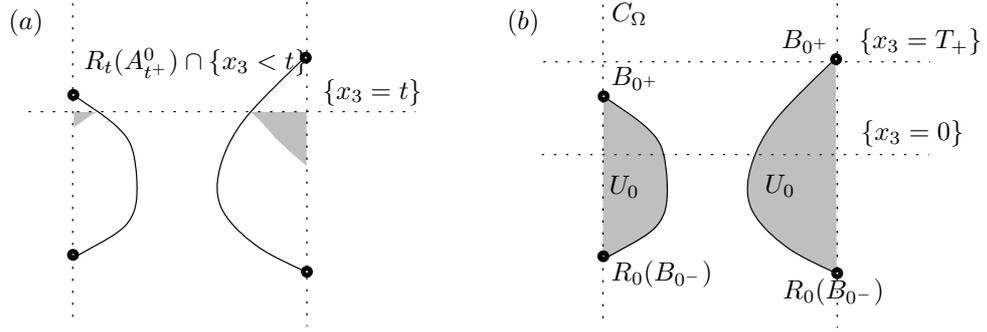\caption{{\small (a) An illustration of (P4): The set $R_t(A^0_{t^+})\cap\{x_3<t\}$, depicted in grey, is contained in $U^0$. (b) An illustration of (P5): Only $U_0$ and $U_1$ intersect $\{x_3=t\}$.}}\label{fig uit}
\end{figure}
Figure \ref{fig uit}. Let us consider the set of heights
\[
G=\set{t\in (0,T_+): \mbox{properties (P1)--(P5) hold for $t$}}\,,
\]
 where the properties defining $G$ are:
	\begin{enumerate}
		\item[(P1)] $\Sigma\cap \{x_3=t\}$  is a subset of ${\rm reg}(\Sigma)$;
		\item[(P2)] $|\nabla_{\Sigma} x_3|<1$ on $\Sigma\cap \{x_3=t\}$;
		\item[(P3)] $R_t(\Sigma_{t^+})$ and $\Sigma_{t^+}$ are graphs with locally bounded slope over $\Omega_t=\pi(\Sigma_{t^+})$;
		\item[(P4)] $\overline{R_t(A^0_{t^+})}\cap \set{x_3<t}\cap C_{\Omega}\subset U^0$;
		\item[(P5)] The only cells of $\mathcal{C}$ whose closures meet $\{x_3=t\}$ are $U^0$ and $U^1$.
	\end{enumerate}
We claim that $G=(0,T_+)$.  To prove this, we show that
\begin{equation}
  \label{ex of t0}
  \mbox{$\exists\, t_0\in(0,T_+)$ such that $(t_0,T_+)\subset G$}\,,
\end{equation}
and then prove that one may take $t_0=0$.

First of all, keeping in mind  we excluded the trivial case \eqref{trivial case}, one has
	\begin{equation}
		\label{nabla sigma x3 near top}
		\mbox{$\exists\, t_0\in(0,T_+)$ such that}\,\,
\left\{\begin{split}&\mbox{(P1) holds for all $t\in [t_0,T_+)$, and}
\\
&\mbox{$0<|\nabla_{\Sigma} x_3|<1$ on $\Sigma_{t_0^+}$}\end{split}\right .\,\,
	\end{equation}
To see this we first observe that
\begin{equation}
  \label{ep}
  \set{x_3=T_+}\cap \Sigma\cap C_\Omega=\emptyset\,.
\end{equation}
Indeed, let $p\in \set{x_3=T_+}\cap \Sigma\cap C_\Omega$, set $U=C_\Omega$, $V=C_\Omega\cap\{x_3>T_+\}$, $\Sigma_2={\Omega}\times\{T_+\}$ and denote by $\Sigma_1$ the component of $\Sigma\cap C_{\Omega}$ containing $p$. Lemma \ref{UniqueContLem} and \eqref{convex hull principle} imply that $p\in \Sigma_1\cap \Sigma_2$ and so $\Sigma_1=\Sigma_2={\Omega}\times\{T_+\}$. That is, \eqref{trivial case} holds, contradicting the assumption made in step two.

\smallskip

By \eqref{ep}, $\Sigma\cap\{x_3=T_+\}\subset B$ and, by definition, $\Sigma$ is regular in a neighborhood of $B$, thus, for $t_0$ closed enough to $T_+$, (P1) holds for every $t\in[t_0,T_+)$. In particular,
\begin{equation}
  \label{Q not t0}
  Q\cap\{t_0\le x_3\}=\emptyset.
\end{equation}
Now, let $p\in \Sigma\cap \{x_3=T_+\}\subset B$. If $|\nabla_{\Sigma} x_3|(p)=0$, then \eqref{convex hull principle} and the Hopf maximum principle applied to $\Sigma$ and $\overline\Omega\times\{T_+\}$ imply there is a connected neighborhood, $\Sigma_p$, of $p$ in $\Sigma$ so $\Sigma_p\subset \Sigma\cap \{x_3=T_+\}$. If $\Sigma_1$ is the component of $\Sigma\cap C_{\Omega}$ containing $\Sigma_p$ and $\Sigma_2=\Omega\times \set{T_+}$, then  $\Sigma_1\cap \Sigma_2\neq \emptyset$ and so, arguing as above, \eqref{trivial case} holds, and a contradiction is reached. Therefore $|\nabla_{\Sigma} x_3|>0$ on $\Sigma \cap\set{x_3=T_+}$, and so \eqref{nabla sigma x3 near top} holds for $t_0$ near enough to $T_+$ by continuity.

\smallskip

Again, by continuity, to show $|\nabla_{\Sigma} x_3|<1$ on $\Sigma_{t_0^+}$ for $t_0$ near to $T_+$ it is enough to show $|\nabla_{\Sigma} x_3|<1$ on $\Sigma\cap\{x_3=T_+\}\subset B$. To show this last fact, let $H$ be a supporting closed half-space to $C_{\Omega}$ at $p$ ($H$ is unique as $\partial \Omega$ is $C^1$ regular), and set $\Pi=\partial H$. By \eqref{convex hull principle}, $\Sigma\subset H$.  Consider the half-space $T_p\Sigma$. By the Hopf maximum principle, if $T_p\Sigma \subset \Pi$, then there is a neighborhood $\Sigma'$ of $p$ in $\Sigma$ with $\Sigma'\subset \Pi$.  As $\Pi \cap C_{\Omega}=\emptyset$, contradicts assumption (e), i.e., that $\Sigma\backslash B\subset C_{\Omega}$. Hence, $T_p \Sigma \subsetneq \Pi$, while $B\subset\pa C_\Om$ implies $T_pB\subset\Pi$. Since $T_pB$ is not vertical (either by assumption (a), or because, in this specific case, it is actually contained into $\{x_3=T_+\}$, and thus is horizontal), we conclude that $|\nabla_{\Sigma} x_3|(p)<1$ and so \eqref{nabla sigma x3 near top} holds.  In fact, as we will use later, this argument implies
\begin{equation}
\label{P2Boundary} |\nabla_\Sigma x_3|<1 \mbox{ for any $p\in B\cap \set{x_3>0}$}.
\end{equation}
\smallskip

We now show that, after possibly moving $t_0$ toward $T_+$, (P1)-(P5) hold for $t\in(t_0,T_+)$. Indeed, \eqref{nabla sigma x3 near top}, immediately gives a $t_0$ so (P1) and (P2) hold for every $t\in (t_0,T_+)$. Up to moving $t_0$, this implies (P3) holds for $t\in (t_0,T_+)$. In particular,
\begin{equation}
  \label{graph t0}
  \Sigma\cap\{x_3>t_0\}=\mbox{graph of a smooth function over $\pi(\Sigma\cap\{x_3>t_0\})$}\,.
\end{equation}
By \eqref{cell structure2} and \eqref{Q not t0}, we have that
\begin{equation}
\label{cell structure2 star}
\{x_3>t_0\}\cap(C_{\Omega}\backslash \Sigma)=\{x_3>t_0\}\cap\bigcup_{i=0}^N U^i\,.
\end{equation}
By the convex hull property, each component of  $\{x_3>t_0\}\cap(\bar{C}_{\Omega}\backslash \Sigma)$ must intersect $\partial C_\Omega$.  Hence, it follows from assumptions (c) and (f) that $\{x_3>t_0\}\cap(C_{\Omega}\backslash \Sigma)=\{x_3>t_0\}\cap\left(U^0\cup U^1\right)$. Hence,  as \eqref{trivial case} does not hold and $\Sigma$ defines a cell structure in $C_{\Omega}\cap \set{x_3>t_0}$
\begin{eqnarray*}
 \Sigma_{t^+}&=& \partial A_{t^+}^0\cap \partial A_{t^+}^1\,,
\\
\set{T_+>x_3> t}\cap \bar{C}_{\Omega}&=&	\set{T_+>x_3> t} \cap \left(\Sigma_{t^+}\cup {A}^0_{t^+}\cup {A}^1_{t^+}\right)\,,
\end{eqnarray*}
for every $t\in(t_0,T_+)$. This immediately implies, that after moving $t_0$ toward $T_+$ by any amount, (P5) holds for $t\in (t_0,T_+)$.  Moreover, combining this with \eqref{nabla sigma x3 near top} implies that, possibly up to further moving $t_0$ toward $T_+$, (P4) hold for $t\in(t_0,T_+)$ -- see Figure \ref{fig uit}.

	\medskip
	
\noindent {\it Step five }: We show that $G=(0,T_+)$. Suppose instead that
\[
t_1=\sup\set{ t<T_+: t\not\in G}>0\,.
 \]
We prove that $t_1\not\in G$ by showing that $[t_1,T_+)\subset G$ implies the existence of $\delta>0$ such that $(t_1-\delta,T_+)\subset G$. By continuity, it is clear that if (P1) and (P2) hold at $t=t_1$, then they hold whenever $|t-t_1|<\delta$ for some $\delta>0$. The implicit function theorem, the validity of (P1) and (P2) for $|t-t_1|<\delta$ and the fact that (P3) already holds for $t\in [t_1,T_+)$, together imply that,  up to decreasing, $\delta$, (P3) holds for $t\in(t_1-\delta, T)$. Finally, the argument used above to deduce that (P4) and (P5) hold on $(t_0,T_+)$ from the fact that (P1), (P2) and (P3) hold on $(t_0,T_+)$ can be repeated verbatim with $(t_1-\delta,T_+)$ in place of $(t_0,T_+)$.

\medskip

We have thus proved that $t_1\not\in G$: in particular, (P1)--(P5) hold for every $t\in(t_1,T_+)$, but at least one of them fails at $t=t_1>0$. We now exclude these five possibilities to reach a contradiction. This will ultimately prove that we cannot have $t_1>0$, and thus that $t_1=0$ and so $G=(0,T_+)$.  First, we show there is no infinitesimal symmetry at $t=t_1$ when $t_1>0$.
	
\medskip

\noindent {\it Proof there is no infinitesimal symmetry at $t=t_1>0$}:  It is true that
\begin{equation}\label{nosymmetry}
 \mbox{if $t_1>0$, $p\in \Sigma\cap \set{x_3=t_1}\backslash Q$ and $T_p\Sigma\neq \set{x_3=0}$, then $R_0(T_p \Sigma)\neq T_p\Sigma$}.
\end{equation}
	We argue by contradiction and suppose $R_0(T_p\Sigma)=T_p\Sigma$.  As $T_p\Sigma \neq \set{x_3=0}$, $T_p\Sigma \cap \set{x_3>0}$ is non-empty.  Hence, there is a component, $\Sigma^+$ of $\Sigma_{t_1^+}^\circ$ so that $p \in \bar{\Sigma}^+$.
	As (P1) holds for $t>t_1$, $\Sigma^+$ is regular in $\set{x_3>t_1}\cap C_{\Omega}$. Set $V'=U^0\cap \set{t_1<x_3<T_+}\subset A_{t_1^+}^0$ so $V'$ is open in $C_{\Omega}\cap \set{x_3>t_1}$. As (P4) holds for $t>t_1$, there is a connected component $V$ of $V'$ so that
	$\Sigma^+\subset\partial V$  in $C_{\Omega}\cap \set{x_3>t_1}$. Moreover, as $V'=\bigcup_{t>t_1} A_{t^+}^0$, the fact that (P5) holds for $t>t_1$ implies $R_{t_1}(V')\subset U^0$ and so $R_{t_1}(V)\cap \Sigma=\emptyset$. That is, the hypotheses of Lemma \ref{SymmetryLem} hold in $U=C_\Omega$.  Hence,
	\begin{equation}\label{xmas}
	R_{t_1}(\Sigma^+)\cap C_\Omega\subset \Sigma \mbox{ and so, as $\Sigma$ is closed, } R_{t_1}(\Sigma^+) \subset \Sigma.
	\end{equation}
	By the convex hull principle for stationary varifolds, $B\cap \Sigma^+\neq \emptyset$ and so there is a $q\in B\cap \Sigma^+$.  Observe that as $\Sigma^+\subset \Sigma_{t_1}^\circ$, $x_3(q)>t_1$. By hypotheses (e), $R_0(\Sigma_+)\cap \partial C_\Omega\subset B$ and so  \eqref{xmas} implies $$
	\set{q,R_{t_1}(q)}\subset  \pi^{-1}(\pi(q))\cap \Sigma\subset B.$$
	 If $t_1\geq \frac{1}{2}T_+$ this implies $q,R_{t_1}(q)\in B_{0^+}$ contradicting $ B_{0^+}$ being a graph.  If $t_1\in (0, \frac{1}{2}T_+) $, then $R_{t_1}(q)\in B_{0^-}$ and $x_3(R_{t_1}(q))>x_3 (R_0(q))$, a contradiction to $R_0(B_{0^+})\geq B_{0^-}$, i.e., (b). From this we conclude that \eqref{nosymmetry} holds.
	
	\medskip

\noindent {\it Proof that $t_1\not\in G$ and $t_1>0$ imply (P1) holds at $t=t_1$}:  If (P1) fails at $t=t_1$, then
 \begin{equation}
   \label{proof of P1}
    \exists\,p\in \sing(\Sigma)\cap \{x_3=t_1\}\,.
 \end{equation}
 As $\sing(\Sigma)\cap B=\emptyset$, $t_1<T_+$, and $Q$ is a finite set of points, there is $R>0$ such that $\Sigma$ is a minimal Plateau surface without boundary in $B_R(p)\backslash \set{p}$.  Moreover, by (P3) and $(t_1, T_+)\subset G$,  one has that $ \Sigma\cap B_R\cap \set{x_3>t_1}$ is a graph of locally bounded slope. Since $\bar\Theta(\Sigma,p)<2$ and $p\in\sing(\Sigma)$, Lemma \ref{RemoveSingLem} implies that $\Sigma$ is a minimal Plateau surface in $B_R(p)$ and that $p$ is a $Y$-point of $\Sigma$, with the spine of the tangent $Y$-cone $T_p\Sigma$ lying in the horizontal plane $\set{x_3=0}$. Thus, up to rotating $\Sigma$ around the $x_3$-axis,
\begin{equation}
  \label{t1t2t3} 	T_p\Sigma=H_{\theta_1}\cup H_{\theta_2} \cup H_{\theta_3}, \mbox{ where }
  |\theta_2| \leq 30, \; \theta_1=\theta_2+120, \; \theta_3=\theta_2-120;
\end{equation}  see
\begin{figure}
  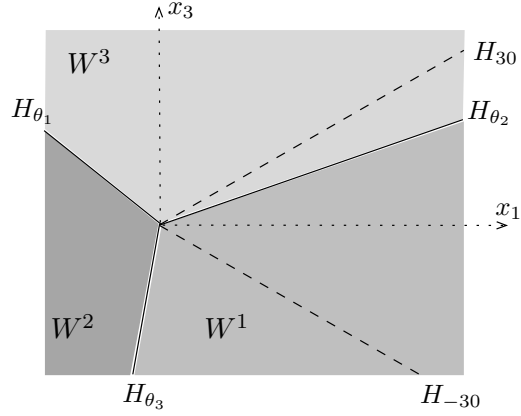\caption{{\small The half-plane $H_\theta$ is defined by a rotation of $H=H_0=\{x_3=0\,,x_1\ge0\}$ by an angle $\theta$ around the $x_2$ axis, with the convention that $H_{\theta_2}$ in this picture corresponds to $\theta_2>0$.}}\label{fig t1t2t3}
\end{figure}
	Figure \ref{fig t1t2t3}. Let $W^3$ be the region of $\Real^3\backslash T_p \Sigma$ between $H_{\theta_1}$ and $H_{\theta_2}$ and likewise let $W^2$ be the region between $H_{\theta_1}$ and $H_{\theta_3}$ and $W^1$ the region between $H_{\theta_2}$ and $H_{\theta_3}$. Appealing to Lemma \ref{CellularLem}, let $\mathcal{C}(T_p\Sigma)=\set{W^1, W^2,W^3}$ and let $U^{i_j}=\mathcal{I}_p(W^j)$ be the cells in  $\mathcal{C}(\Sigma)$ that correspond to $W^j$,  $j=1,2,3$. By Lemma \ref{CellularLem}, $\mathcal{I}_p$ is injective and so $U^{i_j}\neq U^{i_k}$ for $j\neq k$.

 We claim $\theta_2=0$, $\theta_1=120$, $\theta_3=-120$. Suppose $\theta_2>0$.  In this case, $H_{\theta_1}$ and $H_{\theta_2}$ both meet $\set{x_3>0}$, and so $W^1, W^2$ and $W^3$ all meet $\set{x_3>0}$ (this is exactly the situation depicted in Figure \ref{fig t1t2t3}).  As (P5) holds for $t\in(t_1,T_+)$, this means that $\set{U^{i_1}, U^{i_2}, U^{i_3}}=\set{U^0, U^1}$ which is impossible as the three regions must be distinct.  Hence, $\theta_2\leq 0$, see
\begin{figure}
  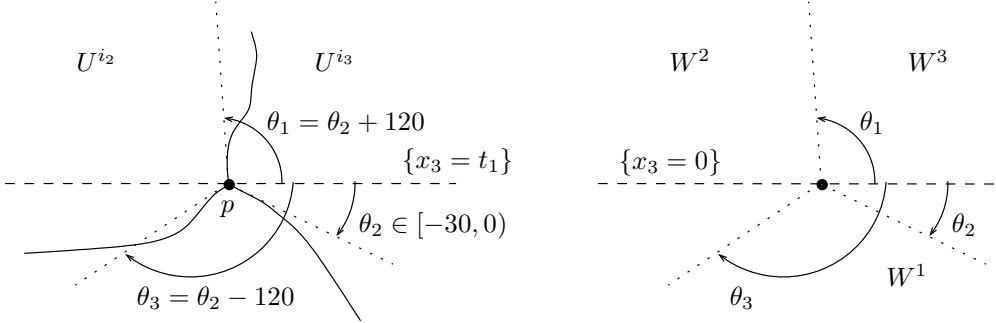\caption{{\small The situation, in a sufficiently small ball near $p$, when $\theta_2\le 0$. The validity of $R_t(A^0_{t^+})\cap\{x_3   <t\}\cap C_\Omega\subset U^0$ for every $t\in(t_1,T_+)$ implies that, if $U^{i_k}=U^0$, then $R_0(W^k)\cap\{x_3<0\}\subset W^k$.}}\label{fig ui2ui3}
\end{figure}
Figure \ref{fig ui2ui3}. In this case, only $W^2$ and $W^3$ intersect $\set{x_3>0}$ and so, as (P5) holds for $t\in (t_1,T_+)$, $\set{U^{i_2}, U^{i_3}}=\set{U^0, U^1}$. Thus, there are two cases,
\begin{eqnarray}\label{ui2ui3 cases}
  \mbox{either $U^{i_2}=U^1$ and $U^{i_3}=U^0$}\,,\qquad\mbox{or $U^{i_2}=U^0$ and $U^{i_3}=U^1$}\,.
\end{eqnarray}
The validity of (P4) for $t\in (t_1,T_+)$ implies that either in the first case of  \eqref{ui2ui3 cases} that $R_0(W^3) \cap\{x_3<0\}\subset W^3$ and in the second that $R_0(W^2) \cap\{x_3<0\}\subset W^2$.
In the first case,  $-\theta_1\ge \theta_2$ while, by \eqref{t1t2t3}, $\theta_2+120=\theta_1\le-\theta_2$ and so $\theta_2\le-60$. This contradicts $\theta_2\in[-30,0]$ and so does not occur. In the second case, $-\theta_1\le \theta_3$, and so combined with \eqref{t1t2t3} one has $0\leq \theta_1+\theta_3=2\theta_2\leq 0$ and so $2\theta_2=\theta_1+\theta_3=0$. This verifies the claim that $\theta_2=0$, $\theta_1=-120$, and $\theta_3=120$.  In particular, $R_0(T_p\Sigma)=T_p\Sigma$, however this contradicts \eqref{nosymmetry} and so (P1) holds at $t=t_1$.

\medskip
	
\noindent {\it Proof that $t_1\not\in G$ and $t_1>0$ imply that (P2)--(P5) holds at $t=t_1$}: If (P2) does not hold for $t=t_1$, then, by \eqref{P2Boundary}, there  is a point $p\in  \set{x_3=t_1} \cap \Sigma$, $p\not\in B$ such that $|\nabla_\Sigma x_3|(p)=1$ --  recall, $\Sigma\cap\{x_3=t_1\}$ consists of regular points as (P1) has already been established at $t=t_1$. Hence, $T_p\Sigma$ is vertical and so $T_p\Sigma\neq \set{x_3=0}$ and $R_0(T_p\Sigma)=T_p\Sigma$.  As this contradicts \eqref{nosymmetry}, (P2) must hold for $t=t_1$.  (P3) follows immediately from (P1), (P2) and the fact that (P3) holds for $t>t_1$.
	
 If (P4) holds for $t\in (t_1,T_+)$ but fails at $t=t_1$, one must have that
\[
\overline{R_{t_1}\big(A^0_{t_1^+}\big)}\cap \set{x_3<t_1}\cap C_{\Omega}\subset \overline{U}^0\,,
\]
holds but
\[
\overline{R_{t_1}\big(A^0_{t_1^+}\big)}\cap \set{x_3<t_1}\cap C_{\Omega}\subset U^0
\]
does not. Therefore, there is $p\in\partial U^0\cap \partial( R_{t_1}(A^0_{t_1^+}))\cap \set{x_3<t_1}\cap C_{\Omega}$. Since
\[
\partial\big( R_{t_1}(A^0_{t_1^+})\big)\cap \set{x_3<t_1}\subset R_{t_1}\big(\Sigma_{t_1^+}\big)
\]
and (P1) holds for $t\geq t_1$, we see that $p$ is a regular point of $R_{t_1}(\Sigma_{t_1^+})$.  However, as $p\in \partial U^0\cap C_{\Omega}$, we also have $p\in \Sigma$ and so applying Lemma \ref{UniqueContLem}, gives $R_{t_1}(\Sigma_{t_1^+})\subset \Sigma$ and this yields a contradiction as in the proof of \eqref{nosymmetry}.  Hence, (P4) holds at $t=t_1$.

\medskip

Finally, if (P5) fails for $t=t_1$, we can find $U^k$ with $k\ne 0,1$ such that $\bar{U}^k\cap\{x_3=t_1\}\ne\emptyset$. Up to relabeling, we can set $k=2$, and thus consider the existence of $p\in\bar{U}^2\cap\{x_3=t_1\}$. By assumption (c), $\bar{U}^2\cap (C_\Omega)_{0^+}=\emptyset$ and so $p\in C_\Omega$. Moreover, the validity of (P5) for $t>t_1$ implies that $\bar{U}^2\subset \set{x_3\leq t_1}$ and so $p\in \partial U^2\cap C_{\Omega}\subset \Sigma$. In fact, as (P1) holds at $t=t_1$,  $p\in \mathrm{int}(\Sigma)$.  Given that $\Sigma$ agrees with $\partial U^2$ near $p$, one has  $\nabla_{\Sigma} x_3(p)=0$. Hence, the strict maximum principle implies $x_3=t_1$ on $B_{r}(p)\cap \Sigma$ for some small $r>0$.  As (P1) is an open condition, there is $\delta>0$ so that $\Sigma_1=\Sigma\cap\{x_3>t_1-\delta\}$ is a regular minimal surface with boundary in $\set{x_3>t_1-\delta}\cap \bar{C}_\Omega$.  In particular, we may apply the standard unique continuation principle for smooth minimal surfaces to $\Sigma_1$ and the connected surface $\Sigma_2=\Omega\times\set{t_1}$ to see that $\Sigma_2\subset \Sigma_1\subset \Sigma$. This contradicts \eqref{trivial case} and so conclude (P5) holds at $t=t_1$.  Hence, if $t_1\not\in G$ and $t_1>0$, then $t_1\in G$. and so $t_1=t_0=0$ and $G=(0,T_+)$.
	
\medskip
	
	\noindent {\it Step Six}: To conclude the proof we first observe that $G=(0,T_+)$ immediately implies (i) and (ii) hold.   We are left to show conclusion (iii), namely the existence of $\epsilon>0$ so $\Sigma$ is a minimal Plateau surface in $\set{x_3>-\epsilon}$. As $G=(0,T_+)$, $\Sigma\cap \set{x_3>0}$ is a regular minimal surface with boundary, so we need only check that $Q\cap\{|x_3|<\epsilon\}=\emptyset$ for a suitable $\epsilon>0$. As $Q$ is a finite set contained in $C_\Omega$, we only need to check that if $p\in \{x_3=0\}\cap \Sigma\cap C_{\Omega}$, then $p\not\in Q$. Clearly, there is an $r>0$ so that $\Sigma$ is a minimal Plateau surface without boundary in $B_r(p)\setminus\{p\}$. Obviously $\bar\Theta(\Sigma,p)<2$, and since $G=(0,T_+)$, $\Sigma\cap\{x_3>0\}$ is a graph of locally bounded slope. By Lemma \ref{RemoveSingLem}, $p$ is either a regular or a $Y$-point, so it does not belong to $Q$, as claimed.
\end{proof}

\section{Rigidity for minimal Plateau surfaces in a slab}\label{section rigidity slab}
In this section we prove the rigidity of minimal Plateau surfaces in a slab with symmetric convex boundary.  We begin by proving topological rigidity in Proposition 3.1, which consists in showing that such minimal Plateau surfaces are {\it simple} bi-graphs. Combined with the previous section and a moving planes argument of Pyo \cite{Pyo} this will complete the proof Theorem \ref{MainThm}. This topological rigidity is an extension of an argument of Ros \cite{Ros} to minimal Plateau surfaces.  Note that Ros's argument uses the Lopez-Ros deformation \cite{LopezRos} and so is special to $\Real^3$.
\begin{prop}[cf. Theorem 1 of \cite{Ros}] \label{LopezRosProp}
  Let $\Sigma\subset\set{|x_3|\leq 1}$ be a connected minimal Plateau bi-graph with $\partial \Sigma=\Gamma\times \set{\pm 1}$ where $\Gamma\subset \Real^2$ is convex.  If $\Sigma$  is symmetric across $\set{x_3=0}$ and $\Sigma$ defines a cell structure in $\set{|x_3|<1}$, then $\Sigma$ is simple.
\end{prop}
\begin{proof}
	Let $\Sigma_+=\bar{\Sigma}_{0^+}^\circ=\overline{\Sigma \cap \set{x_3>0}}$.  The symmetry of $\Sigma$ and the fact that it is a bi-graph implies that $\Sigma_+$ is a regular minimal surface with boundary whose interior is a graph over $\set{x_3=0}$.  In particular, as $\Sigma$ is connected, $\Sigma_+$ is a connected planar domain.
	One readily checks that at the points of $\set{x_3=0}\cap \partial\Sigma_+$, $\Sigma$ either intersect $\set{x_3=0}$ orthogonally (if the point is a regular point of $\Sigma$) or intersect at an angle of $120^\circ$ (if the point is a $Y$-point of $\Sigma$). There must exist such points as $\Sigma$ is connected. In fact,
	as $\Sigma$ defines a cell structure in $\set{|x_3|<1}$ one must have either $\set{x_3=0}\cap \Sigma \subset \reg(\Sigma)$ or $\set{x_3=0}\cap \Sigma = \sing(\Sigma)$ -- see  Figure \ref{fig regularbi}-(b).  That is, either $\Sigma$ is regular or every component of $\Sigma_+\cap \set{x_3=0}$ consists of $Y$-points and bounds a disk in $\Sigma$.
	
	 If $\Sigma$ is regular, then this means $\Sigma_+$ solves the  free boundary Plateau problem for the data $(\Gamma_+, \set{x_3=0})$ in the sense of \cite{Ros} and so is an annulus by \cite[Corollary 3]{Ros}. It immediately follows that $\Sigma$ is also an annulus and so is simple.
	
	 If $\Sigma$ is singular, then  the constant contact angle with $\set{x_3=0}$, continues to imply that every non-null homologous loop in $\Sigma_+$ has vertical flux.  Indeed, let $\sigma$ be an (oriented) component of $\partial \Sigma_+$ that meets $\set{x_3=0}$ at $120^\circ$.  Let $\nu_\sigma$ is the outward conormal to $\sigma$ in $\Sigma_+$ and let $\mathbf{n}_{\sigma}$ be the outward normal to $\sigma$ in $\set{x_3=0}$.   Clearly, $\nu_\sigma(p)=-\frac{1}{2}\mathbf{n}_{\sigma}(p)-\frac{\sqrt{3}}{2} \mathbf{e}_3$ and so,
	 $$
	 \mathrm{Flux}(\sigma)= \int_{\sigma} \nu_{\sigma} d\mathcal{H}^1=-\frac{1}{2} \int_{\Sigma} \mathbf{n}_{\sigma} d\mathcal{H}^1-\frac{\sqrt{3}}{2} \mathcal{H}^1(\sigma) \mathbf{e}_3=-\frac{\sqrt{3}}{2} \mathcal{H}^1(\sigma) \mathbf{e}_3
	 $$
	 where the last equality follows from applying the divergence theorem in $\set{x_3=0}$. As any closed curve in $\Sigma_+$ is homologous to some linear combination of the components of $\set{x_3=0}\cap \partial \Sigma_+$, it follows that $\Sigma_+$ has vertical flux for each closed curve.
	
	 To complete the proof we use the Lopez-Ros deformation \cite{LopezRos} to reduce to the regular case.  To that end consider the Weierstrass data $(M, \eta, G)$ of $\Sigma_+$.  Here $M$ is the underlying Riemann surface structure of $\Sigma_+$, $\eta$ is the (holomorophic) height differential (i.e., the complexification of $dx_3$) and $G$ is the meromorphic function given by the stereographic projection of the Gauss map (of the outward normal).  This data produces a conformal embedding  of $\Sigma_+$ by $M$
	 $$\mathbf{F}:M\to \Sigma^+\subset \Real^3$$ given by
	 $$
	 \mathbf{F}(p)=\mathrm{Re} \int_{p}^{p_0}\left(\frac{1}{2}(G^{-1} -G), \frac{i}{2}(G^{-1}+G), 1  \right)\eta.
	 $$  Let $\partial_+ M$ be the component of $\partial M$ sent to $\Gamma\times \set{1}$ and let $\partial_-M=\partial M\backslash \partial_+M$ be the components sent to $\Sigma_+\cap \set{x_3=0}$.   As $\Sigma_+$ meets $\set{x_3=0}$ at $120^\circ$, one has $|G|=\gamma_0=\frac{\sqrt{3}}{3}> 0$, is constant on $\partial_-M$. As observed by Lopez-Ros \cite{LopezRos}, because the flux of $\Sigma_+$ is vertical, the Weierstrass data $(M, \eta, \gamma_0^{-1} G)$ produces a conformal immersion $\mathbf{F}':M\to \Upsilon_+\subset \Real^3$ of a new (possibly immersed) minimal surface with boundary $\Upsilon_+$.  The properties of the Lopez-Ros deformation ensure that $\partial \Upsilon_+\subset \set{x_3=1}\cup \set{x_3=0}$ and $\mathbf{F}(\partial_+M)= \partial \Upsilon_+\cap \set{x_3=1}$ is convex -- see \cite[Lemma 2]{PerezRos} while $\Upsilon_+$ meets $\set{x_3=0}$ orthogonally.  It follows that the set $\Upsilon=\Upsilon_+\cup R_0 (\Upsilon_+)$ given by taking the union of $\Upsilon_+$ with its reflection across $\set{x_3=0}$ gives a connected smooth minimal (possibly immersed) surface whose boundaries are convex curves lying on $\set{x_3=\pm 1}$.   By a result of Ekholm, Weinholtz and White \cite{EWW}, $\Upsilon$ is embedded and so $\Upsilon_+$  solves the  free boundary Plateau problem for the data $(\Upsilon_+, \set{x_3=0})$ in the sense of \cite{Ros} and so, as before, is an annulus by \cite[Corollary 3]{Ros}.  Hence, $\Sigma_+$ is an annulus and so $\Sigma$ is also simple in the singular case.	
\end{proof}

We are now in a position to prove Theorem \ref{MainThm}.  For brevity we use Proposition \ref{LopezRosProp} to allow us to appeal to a result of Pyo \cite{Pyo} to handle the case where the boundaries are circles, however, one could also work directly with moving planes argument used in \cite{Pyo} and avoid Proposition \ref{LopezRosProp}.

\begin{proof}[Proof of Theorem \ref{MainThm}]
 Let $\Omega\subset \Real^2$ be the convex open domain so $\Gamma=\partial \Omega$. We first prove that $\Sigma$ is a simple minimal Plateau bi-graph, which is symmetric by reflection through $\{x_3=0\}$. This is immediate if $\Sigma$ is disconnected. Indeed, in that case, by the convex hull property we find that $\Sigma\subset \set{|x_3|=\pm 1}$, and so $\Sigma=\Omega_-\cup \Omega_+$ where $\Omega_\pm =\Omega\pm \mathbf{e}_3$. We thus assume that $\Sigma$ is connected, and claim that $\Sigma$ satisfies the hypotheses of Corollary \ref{CylCor} in $C_{\Omega}$ with $B=\Gamma_-\cup \Gamma_+$ and $Q=\emptyset$. Indeed, the only item that is not immediate is $\Sigma\backslash B\subset C_{\Omega}$. but this follows from the maximum principle of Solomon-White \cite{solomonwhite} applied to $V_\Sigma$, the varfiold associated to $\Sigma$, and appropriate catenoidal barriers.  Hence, by Corollary \ref{CylCor}, $\Sigma$ is a minimal Plateau bi-graph that is symmetric with respect to reflection across $\set{x_3=0}$.  As $\Sigma$ is connected we may then appeal to Proposition \ref{LopezRosProp} to see that $\Sigma$ is simple.

 Finally, we treat the case that $\Gamma$ is a circle.   To that end,  let $\Sigma_+=\overline{\Sigma\cap \set{x_3>0}}$.  As already observed this set is a regular minimal annulus with one boundary a circle in the plane $\set{x_3=1}$ and the other boundary meeting $\set{x_3=0}$ in a constant contact angle (either $90^\circ$ or $120^\circ$).  It now follows from the main result of \cite{Pyo} that $\Sigma_+$ is a piece of a catenoid.  As such, $\Sigma$ is either a subset of $\Cat$ or of $\Cat_Y$ depending on its regularity.
\end{proof}

\section{Global rigidity of minimal Plateau surfaces with two regular ends}\label{section global rigidity}
In this section we prove Theorem \ref{GlobalThm}.  To do so we first establish certain elementary properties of the ends -- specifically that asymptotically they are parallel and have  equal, but opposite, logarithmic growth rate -- this is entirely analogous to what is done in the regular case.  As a consequence, we may appeal to Theorem \ref{CylThm} to conclude that $\Sigma$ is, after rotation and vertical translation, symmetric with respect to reflection across $\set{x_3=0}$ and that $\Sigma_+=\Sigma\cap \set{x_3\geq 0}$ is a graph of locally bounded slope.  We conclude the proof by using complex analytic arguments -- specifically a variant of the Lopez-Ros deformation \cite{LopezRos} -- to reduce to the case already considered by Schoen \cite{SchoenSymmetry}.

\medskip

We remark that one could conclude by applying the moving planes method with planes moving orthogonally to $\set{x_3=0}$. Indeed, thanks to Theorem \ref{MainThm}, $\Sigma\cap\{x_3>0\}$ is a smooth graph which meets $\{x_3=0\}$ at a constant angle (and $\Sigma\cap\{x_3<0\}$ is just the reflection of $\Sigma\cap\{x_3>0\}$ along $\{x_3=0\}$): therefore we can apply the same  ``horizontal'' moving planes arguments as in \cite{SchoenSymmetry} and \cite{Pyo} to $\Sigma\cap\{x_3>0\}$, and give a direct PDE proof of its rotational symmetry which entirely avoids complex analytic methods.

\begin{defn}
  \label{def regular end} Following \cite{SchoenSymmetry}, we say that a minimal Plateau surface $\Sigma\subset\Real^3$ {\bf has two regular ends}  if there is a compact set $K\subset \Real^3$ so that
$$
\Sigma\backslash K= \Gamma_1\cup \Gamma_2
$$
where there are rotations $S_1, S_2\in SO(3)$ and a radius $\rho>0$ so for $i=1,2$,
$$
S_i \cdot \Gamma_i=\set{(\mathbf{y}, u_i(\mathbf{y})): \mathbf{y}\in \Real^2\backslash \bar{B}_{\rho}}
$$
and
$$
u_i(\mathbf{y})= a_i \log |\mathbf{y}|+ b_i+ \mathbf{c}_i\cdot \frac{\mathbf{y}}{|\mathbf{y}|^2}+R_i(\mathbf{y})
$$
where
$$
|R_i(\mathbf{y})|+|\mathbf{y}| |\nabla R_i(\mathbf{y})|\leq C |\mathbf{y}|^{-2}.
$$
Let $P_i=S_i (\set{x_3=0})$, be the planes the $\Gamma_i$ are graphs over.  One readily checks that $\lim_{\rho \to 0} \rho \Gamma_i=P_i$, that is, each $\Gamma_i$ is asymptotic to the plane $P_i$.
\end{defn}

\begin{lem}\label{GlobalAuxLem}
	One has $\lim_{R\to \infty} \frac{\mathcal{H}^2(\Sigma\cap B_R)}{\pi R^2}=2$.  In fact, one has $P_1=P_2=P$ and $\lim_{\rho\to 0} \rho\Sigma=P$ in $C^\infty_{loc}(\Real^3\backslash \set{0})$.  If $\Sigma$ is disconnected then $\Gamma=P_1'\cup P_2'$ where $P_i'$ are disjoint planes parallel to $P$.
\end{lem}
\begin{proof}
	It is clear from the definition of regular end that $\lim_{\lambda\to 0} \lambda \Sigma= P_1 \cup P_2$
 in $C^1(\Real^3\backslash \set{0})$.  This proves the first claim.
Suppose that $P_1\neq P_2$ as both $P_1$ and $P_2$ are planes through the origin this means that there is a point $q\in \partial B_2\cap P_1\cap P_2$ so that $D_i=B_1(q)\cap P_i$ are two disks that meet transversely along a line segment.  The convergence of $\rho\Gamma_i$ to $P_i$ as $\rho\to 0$.  Implies that for $\rho$ very small $D_i'(\rho)=\rho \Gamma_i \cap B_{1}(q)$ is a graph over $D_i$ with small $C^1$ norm and so $D_1'(\rho)$ meets $D_2'(\rho)$ transversely along a curve in small tubular neighborhood of $D_1\cap D_2$.  This means that $\rho \Sigma$ is not a Plateau surface in $B_1$ (as the infinitesimal model is the transverse union of two planes) and so this cannot occur under the hypotheses of Theorem \ref{GlobalThm}.  Hence, $P_1=P_2=P$.  The nature of the convergence and standard elliptic regularity implies the convergence may be taken in $C^\infty_{loc}(\Real^3\backslash \set{0})$.
	
	Finally, if $\Sigma$ is disconnected, then, as there are no compact minimal Plateau surfaces without boundary, there are exactly two connected components, $\Sigma_1$ and $\Sigma_2$ of $\Sigma$ corresponding to the ends $\Gamma_1$ and $\Gamma_2$.  Clearly, each $\Sigma_i$ is a minimal Plateau surface and $\lim_{\rho\to 0} \rho\Sigma_i =P_i=P$.  By the monotonicity formula this implies each $\Sigma_i$ is a plane that is parallel to $P$ by definition.
\end{proof}

\begin{proof}[Proof of Theorem \ref{GlobalThm}]
If $\Sigma$ is disconnected, then Lemma \ref{GlobalAuxLem} implies $\Sigma$ is a pair of disjoint parallel planes and we are done. If $\sing(\Sigma)=\emptyset$, then $\Sigma$ is a smooth minimal surface and so \cite{SchoenSymmetry} applies and we are also done.
  As such we may assume $\Sigma$ is connected and $\sing(\Sigma)\neq \emptyset$.  In this, case up to an ambient rotation we may assume the the unique tangent plane at infinity, $P$,  given by Lemma \ref{GlobalAuxLem} is $P=\set{x_3=0}$. Let $u_i$ be the functions with the given asymptotics for the ends $\Gamma_i$. Note that even though $P_1=P_2=P$, there is still a freedom in the choice of the rotations $S_i$.  For concreteness, choose the same rotation for both ends.  As a consequence, by vertically translating $\Sigma$ appropriately, we may assume $b_1+b_2=0$.

It follows from standard calculations (e.g., those in \cite{SchoenSymmetry}) that the flux of each $\Gamma_i$ is vertical.  In fact, if $\sigma_i$ is an appropriately oriented choice of generator for the homology of the annulus $\Gamma_i$, then
$$
\mathrm{Flux}(\sigma_i)= \int_{\sigma_i} \nu_{\sigma_i} d\mathcal{H}^1=2\pi a_i\mathbf{e}_3.
$$
Hence, by the balancing properties of the flux -- which hold for minimal Plateau surfaces as they follow from \eqref{minimal PS} -- one has $2\pi a_1+2\pi a_2=0$.  Up to relabelling, one may assume $a_1\geq 0 \geq a_2=-a_1$.  In fact, by the strong half-space theorem \cite{HoffmanMeeks},  $a_1>0>a_2=-a_1$.

Take $R>1$ large and let $\Sigma_R=\Sigma\cap \bar{C}_{R}$ be the closed cylinder of radius $R$ centered on the $x_3$-axis.  Our assumptions on $\Sigma$ and the properties of the  ends imply that, for any $\epsilon>0$, there is an $R_\epsilon>0$ large so that, for $R>R_\epsilon$, $\Sigma_R-\epsilon \mathbf{e}_3$ satisfies the hypotheses of Theorem \ref{CylThm}.  It follows that $\Sigma\cap \set{x_3>\epsilon}$ and $\Sigma\cap \set{x_3<-\epsilon}$ are both graphs over the plane $\set{x_3=0}$ and each is $\epsilon$ close to reflection across $\set{x_3=0}$ of the other.  Taking $\epsilon\to 0$, it follows that $\Sigma\backslash \set{x_3=0}$ consists of two graphical components and is symmetric with respect to reflection across $\set{x_3=0}$.

To complete the proof one considers $\Sigma_+=\overline{\Sigma\cap \set{x_3>0}}$.  As $\sing(\Sigma)\neq \emptyset$, $\Sigma_+$ is a surface with one catenoidal end that meets $\set{x_3=0}$ along one boundary curve with constant contact angle equal to $120^\circ$. Observe that as $\Sigma_+$ has a catenoidal end, the underlying Riemann surface structure of $\Sigma_+$ is $M\backslash \set{p_0}$ where $M$ is a compact Riemann surface with boundary and $p_0\not\in \partial M$. Let  $(M\backslash \set{p_0}, \eta, G)$ be Weierstrass data for $\Sigma_+$ so $\eta$,  is the height differential, and $G$, the stereographic projection of the Gauss map of the outward pointing normal.   As $\Sigma_+$ has a catenoidal end, $\eta$ and $G$ both extend meromorphically to $M$ with $\eta$ having a simple pole at $p_0$ and $G$ a simple zero.  Moreover, as $\Sigma_+$ meets $\set{x_3=0}$ at $120^\circ$,  $|G|=\gamma_0=\frac{\sqrt{3}}{3}> 0$ on $\partial M$.   As in the proof of Proposition \ref{LopezRosProp}, the constant contact angle implies that the flux over any closed loop in $\Sigma_+$ is vertical.  Hence, by \cite{LopezRos}, the Weierstrass data $(M\backslash \set{p_0}, \eta, \gamma_0^{-1} G)$ parameterizes a new (possibly immersed) minimal surface with boundary $\Upsilon_+$ and this surface also has a regular end asymptotic to a vertical catenoid.  Moreover, $\partial \Upsilon_+\subset \set{x_3=0}$ and, as the boundary of $\Upsilon_+$ is parameterized by $\partial M$, the choice of Weierstrass data ensures $\Upsilon_+$ meets $\set{x_3=0}$ orthogonally.  It follows that the set $\Upsilon=\Upsilon_+\cup R_0 (\Upsilon_+)$ given by taking the union of $\Upsilon_+$ with its reflection across $\set{x_3=0}$ gives a connected smooth minimal (possibly immersed) surface with two regular ends. As \cite{SchoenSymmetry} applies to immersed minimal surfaces, it follows that $\Upsilon$ is a vertical catenoid. As the Lopez-Ros deformation of a vertical catenoid is just a reparamaterization of the original catenoid, it follows that $\Sigma_+$ is also a subset of a vertical catenoid.  From this one immediately concludes that $\Sigma$ is a $Y$-catenoid.
\end{proof}

\section{Further remarks and open questions}\label{Questions}
We conclude with some further remarks and questions about minimal Plateau surfaces in slabs. First of all, we observe how essential to the proof of Theorem \ref{MainThm} is the assumption that at each point in $\Gamma^\pm=(\pa\Om)\times\{\pm1\}$ the surface $\Sigma$ is locally diffeomorphic to half-planes. If relaxing this assumption we may have more rigidity cases (for example, the union between a catenoid bounded by two circles, and one or both the disks bounded by the two circles, would be admissible if in the definition of Plateau minimal surfaces we relax the notion of boundary point to the case when the surface is locally diffeomorphic to a finite union of half-planes) and indeed our argument immediately breaks down. This is something we explore more thoroughly in \cite{VarifoldPaper}, and which motivates the following question:

\begin{ques}
	Is it possible to find a circle $\Gamma$ in $ \set{x_3=0}$ so that if $\Gamma_\pm =\Gamma\pm \eE_3\subset \set{x_3=\pm 1}$, then there is a minimal Plateau surface $\Sigma$ in $\Real^3\backslash \Gamma_-\cup \Gamma_+$ which does not possess  rotational symmetry?
\end{ques}
A plausible candidate surface would be to desingularize the union of an appropriately scaled pieces of ${\rm Cat}$ and ${\rm Cat}_Y$.  Less clear is whether the orientability condition is necessary.  This motivates the following questions:
\begin{ques}
	Fix two curves $\Gamma_0$ and $\Gamma_1$ in parallel planes -- not necessarily convex.  Is there a minimal Plateau surface $\Sigma$ with $\partial \Sigma=\Gamma_0\cup \Gamma_1$ so that $\Sigma$ does not have an associated cell structure?  Even if such examples exist for general choices of curves, does the conclusion of Theorem \ref{MainThm} still hold? I.e., is the cell condition unnecessary in the convex or circular case?
\end{ques}


Theorem \ref{MainThm} applies to ``unstable" minimal Plateau surfaces as well as to the physical ``stable" ones.  It would be interesting to rigorously produce examples of these sorts examples for large classes of curves.  One approach would be to develop a min-max theory in this setting.
\begin{ques}
	Can one produce unstable singular minimal Plateau surfaces that span pairs of convex curves?
\end{ques}
 Alternatively, one could hope to develop a degree theory analogous to the theory developed by Meeks and White to study the space of minimal annuli spanning a pair of convex curves \cite{MWannuli,MWcommhelv}.  In particular, they show that generic pairs of convex curves are spanned by either no minimal annulus or exactly two, one stable and the other unstable. One may ask to what extent this generalizes to minimal Plateau surfaces that are topologically ${\rm Cat}_Y$ -- i.e. an annulus with a disk glued in.
\begin{ques}
  Can one characterize the space of Plateau minimal surfaces that are topologically ${\rm Cat}_Y$ surfaces and span pairs of convex curves?  For generic pairs are there exactly two such surfaces, one stable and one unstable?
\end{ques}

The Convex Curves conjecture of Meeks \cite{MeeksConj} states that the only connected minimal surfaces spanning two convex curves in parallel planes are topological annuli.  One may ask an analogous question in the Plateau setting.
\begin{ques}
	Must a singular  minimal Plateau surface spanning a pair of convex curves be topologically ${\rm Cat}_Y$?  What if the curves are coaxial circles?
\end{ques}
Theorem \ref{MainThm} shows the answer is yes when the curves are vertical translations of one another provided the surface is cellular -- in the smooth setting this is a result of Ros \cite{Ros} and Schoen \cite{SchoenSymmetry}.

Finally, catenoids possess an interesting variational property.  Namely in  \cite{BernsteinBreiner} the authors show that an appropriate piece of the catenoid has the least area among minimal annuli whose boundaries lie in two fixed parallel planes. This was generalized in \cite{choedaniel} who increased the class of competitors to a larger class of (smooth) minimal surfaces of different topological type.  One may ask the same question in the class of singular minimal Plateau surfaces.
\begin{ques}
 Among minimal Plateau surfaces spanning two fixed parallel planes what is the least area singular surface?  Is it an appropriate piece of ${\rm Cat}_Y$?
\end{ques}

\appendix

\section{A rigidity result for geodesic nets}\label{appendix geonets} In the proof of the Removable Singularity lemma for Plateau minimal surfaces, see Lemma \ref{RemoveSingLem}, we have used a rigidity lemma for geodesic nets on the unit sphere whose statement and proof are presented in this appendix. We say that $\Gamma\subset\mathbb{S}^2$ is a {\bf geodesic net} if $\Gamma$  is a finite union $\Gamma=\bigcup_{i=1}^M\gamma_m$ of geodesic arcs $\gamma_i$ in $\mathbb{S}^2$ so that if $p\in\Gamma$, then, setting $I(p)=\{i:p\in\gamma_i\}$, one has that either $\#\,I(p)=1$ and $p\in{\rm int}\,\gamma_i$, or $\#\,I(p)\ge 2$, $p\in\pa\gamma_i$ for each $i\in I(p)$ and
\begin{equation}
  \label{gamma balance condition}
  \sum_{i\in I(p)}\nu^{{\rm co}}_{\gamma_i}(p)=0\,,
\end{equation}
where $\nu^{{\rm co}}_{\gamma_i}$ denotes the outer unit conormal to $\gamma_i$ in $\mathbb{S}^2$ at $p$. Of course, if $\#I(p)\ge 2$, then $\#I(p)\ge 3$. If $\Gamma$ is a geodesic net in $\mathbb{S}^2$, then the multiplicity one, $1$-dimensional varifold $V_\Gamma$ associated to $\Gamma$ is stationary in $\mathbb{S}^2$. Moreover, a cone $K$ in $\R^3$ with vertex at $0$ induces a multiplicity one, $2$-dimensional, stationary varifold $V_K$ in $\R^3$ if and only if $\Gamma=K\cap\pa B_1$ is a geodesic net in $\mathbb{S}^2\equiv\pa B_1$ thanks to \cite{allardalmgren76}. Of course, any finite union of equatorial circles defines a geodesic net. Equatorial circles and $Y$-nets (three equatorial half-circles meeting at two common end-points at $2\pi/3$-angles) are examples of geodesic nets that are also locally length minimizing, in the sense that they minimize $\mathcal{H}^1$ with respect to  Lipschitz deformations with sufficiently small support. The following lemma provides a rigidity statement which allows one to characterize these two length minimizing geodesic nets among all geodesic nets. The proof uses moving equatorial half-circles.

\begin{lem}[Rigidity of geodesic nets]\label{lemma geodesic nets}
  Let $\Gamma$ be a geodesic net in $\mathbb{S}^2$, let $e$ be a unit vector and let $\varepsilon>0$. If $\Gamma$ agrees either with an equatorial circle or with a $Y$-net in the spherical cap $\{x\cdot e>-\varepsilon\}$, then $\Gamma$ is either an equatorial circle or a $Y$-net.
\end{lem}

\begin{proof}
  Without loss of generality let us assume that $e=e_3$, so that $\Gamma\cap\{x_3\ge0\}$ is equal to a equatorial half-circle $\Gamma_0$ contained in $\{x_3\ge0\}$ with endpoints $p_0$ and $-p_0$. In this way $\Gamma\cap\{x_3>-\varepsilon\}$ is either equal to $S_0\cap\{x_3>-\varepsilon\}$ or to $Y_0\cap\{x_3>-\varepsilon\}$, where  $S_0=\Gamma_0\cup(-\Gamma_0)$ is the unique equatorial circle containing $\Gamma_0$ and $Y_0$ is the unique $Y$-net containing $\Gamma_0$.

  \medskip

  Let $\{\Gamma(t)\}_{0\le t\le\pi}$ and $\{\Gamma'(t)\}_{0\le t\le\pi}$ denote the two distinct one-parameter families of equatorial half-circles obtained by rotating by $t$-radians $\Gamma_0$ around the axis defined by its endpoints $\pm p_0$ one clockwise the other counter-clockwise. In particular,  $\Gamma(t)$ and $\Gamma'(t)$ have the same endpoints of $\Gamma_0$, $\Gamma(0)=\Gamma'(0)=\Gamma_0$, and $\Gamma(\pi)=\Gamma'(\pi)=-\Gamma_0$ is the equatorial half-circle antipodal to $\Gamma_0$. By assumption, there are maximal intervals $[0,\delta_0)$ and $[0,\delta_0')$ such that
  \begin{equation}
    \label{gammat delta0}
      \Gamma(t)\cap\Gamma\setminus\{\pm p_0\}=\emptyset\qquad\forall t\in(0,\delta_0)\,,
  \end{equation}
  and such that the same holds for $\Gamma'(t)$ in place of $\Gamma(t)$ if $t\in(0,\delta_0')$. Notice that as $\Gamma$ agrees with either an equatorial circle or a $Y$-net on $\{x_3>-\varepsilon\}$, then either $\delta_0$ or $\delta_0'$ must be strictly larger than $\pi/2$. We assume, without loss of generality, that $\delta_0>\pi/2$.

  \medskip

  If $\delta_0=\pi$ but $\Gamma\cap\Gamma(\pi)\setminus\{\pm p_0\}=\emptyset$, then the validity of \eqref{gammat delta0} for every $t\in(0,\pi)$ implies that $\Gamma\subset W$ where $W$ is wedge given by the intersection of two different closed half-spaces. Therefore $\#I(p_0)\ge 2$ but \eqref{gamma balance condition} cannot hold at $p=p_0$. We deduce that if $\delta_0=\pi$, then $\Gamma\cap\Gamma(\pi)\setminus\{\pm p_0\}\ne\emptyset$. As a consequence, $\Gamma$ is touched by $\Gamma(\pi)$ at an interior point $q$, and locally near $q$ $\Gamma$ lies on one side of $\Gamma(\pi)$ thanks to \eqref{gammat delta0} with $\delta_0=\pi$: by the strict maximum principle we find that, locally near $q$, $\Gamma$ is equal to $\Gamma(\pi)$.  Let $I$ be the component of $\Gamma \cap \Gamma(\pi)$ containing $q$. As $I$ is the intersection of closed sets it is closed.  Moreover, for every $p\in I$, as $\Gamma$ lies on one side of $\Gamma(\pi)$ near $p$ one has $\# I(p)\leq 2$ and so $\Gamma$ is smooth near $p$.  Hence, we may appeal to a unique continuation to see that $I=\Gamma(\pi)$.  That is, $\Gamma(\pi)\subset\Gamma$. We have thus proved that
  \[
  S_0\subset\Gamma\,,\qquad \Gamma\cap H_0=\emptyset\,,
  \]
  where $H_0$ is one of the two open half-spaces bounded by $S_0$. It is easily seen that \eqref{gamma balance condition} and $\Gamma\cap H_0=\emptyset$ imply that $\#\,I(p)=1$ for every $p\in S_0$. In particular, by a covering argument, $\Gamma$ is equal to $S_0$ in an open neighborhood of $S_0$, and since $\Gamma$ is connected, this implies that $\Gamma=S_0$.

  \medskip

  We are left to discuss the case when $\delta_0\in(\pi/2,\pi)$. By the strict maximum principle, the regularity of points of $\Gamma$ lying on a $\Gamma(\delta_0)$ and the unique continuation principle we see that
  \[
  \Gamma_0\cup\Gamma(\delta_0)\subset\Gamma\,,\qquad \Gamma\cap V=\emptyset\quad\mbox{where}\quad V=\bigcup_{t\in(0,\delta_0)}\Gamma(t)\,.
  \]
  The fact that $\delta_0<\pi$ implies that $\Gamma\cap\{x_3>-\varepsilon\}=S_0\cap\{x_3>-\varepsilon\}$ cannot hold. Therefore it must be $\Gamma\cap\{x_3>-\varepsilon\}=Y_0\cap\{x_3>-\varepsilon\}$, which gives $\delta_0=2\pi/3$, $\delta_0'=2\pi/3$, and thus that $Y_0\subset \Gamma$. Now pick let $V'$ denote the smaller wedge bounded by $\Gamma_0$ and $\Gamma'(2\pi/3)$, and notice that similarly $V$ is the smaller wedge bounded by $\Gamma_0$ and $\Gamma(2\pi/3)$. If $q$ is in the interior of $\Gamma(2\pi/3)$, then the fact that $V\cap\Gamma=\emptyset$ combined with \eqref{gamma balance condition} implies that $\Gamma(2\pi/3)$ is equal to $\Gamma$ in a neighborhood of $q$; similarly, $V'\cap\Gamma=\emptyset$ and \eqref{gamma balance condition} imply that $\Gamma'(2\pi/3)$ is equal to $\Gamma$ in a neighborhood of each of its points. Finally, $\Gamma$ and $Y_0$ agree in a neighborhood of $\{\pm p_0\}$ and in a neighborhood of $\Gamma_0$ thanks to $\Gamma\cap\{x_3>-\varepsilon\}=Y_0\cap\{x_3>-\varepsilon\}$, so that, in conclusion, by a covering argument, $\Gamma$ is equal to $Y_0$ in an open neighborhood of $Y_0$. This proves that $\Gamma=Y_0$, as claimed.
\end{proof}

\bibliographystyle{is-alpha}
\bibliography{references}

\end{document}